\documentclass[11pt]{article}

\usepackage{amsmath}
\usepackage{amsfonts}
\usepackage{amsthm}
\usepackage{amssymb}
\usepackage{mathrsfs}
\usepackage{verbatim}
\usepackage{enumerate}
\usepackage{graphicx}
\usepackage{longtable}
\usepackage[all,cmtip]{xy}
\usepackage{upref}
\usepackage[hang,small,bf]{caption}
\usepackage{marginnote}
\usepackage{a4wide}
\usepackage{mathtools,tensor}
\usepackage{xfrac}
\usepackage{color}
\usepackage{bm}
\usepackage{hyperref}
\usepackage{nomencl}
\usepackage{xtab,cuted}

\newtheorem{thm}{Theorem}[section]
\newtheorem{prp}[thm]{Proposition}
\newtheorem{lem}[thm]{Lemma}
\newtheorem{cor}[thm]{Corollary}

\newtheorem*{con*}{Conjecture 2}
\theoremstyle{definition}
\newtheorem{dfn}[thm]{Definition}
\newtheorem{rmk}[thm]{Remark}

\numberwithin{equation}{section}

\newcommand{\T}{{\mathbb T}}
\newcommand{\Z}{{\mathbb Z}}

\newcommand{\R}{{\mathbb R}}
\newcommand{\N}{{\mathbb N}}

\newcommand{\pt}{{\mathrm{pt}}}

\newcommand{\om}{\mathrm{\omega}}

\newcommand{\Vol}{\mathrm{Vol}}

\newcommand{\id}{\mathrm{id}}

\newcommand{\di}{{\mathrm d}}

\newcommand{\Fvol}{{\mathfrak{Vol}}}

\newcommand{\ta}{{\mathrm T}}
\newcommand{\dR}{{\mathrm{dR}}}
\newcommand{\area}{{\mathrm{area}}}
\newcommand{\avg}{{\mathrm{avg}}}
\newcommand{\norm}{{\mathrm{norm}}}

\newcommand{\can}{{\mathrm{can}}}

\newcommand{\p}{\partial}

\newcommand{\x}{\times}

\newcommand{\beq}{\begin{equation}}
\newcommand{\beqn}{\begin{equation}\nonumber}
\newcommand{\eeq}{\end{equation}}
\newcommand{\bea}{\begin{equation}\begin{aligned}}
\newcommand{\bean}{\begin{equation}\begin{aligned}\nonumber}
\newcommand{\eea}{\end{aligned}\end{equation}}

\AtEndDocument{\bigskip{\footnotesize
  \textsc{Ruprecht-Karls-Universit\"at Heidelberg, Mathematisches Institut, Im Neuenheimer Feld 205, 69120 Heidelberg, Germany} \par  
  \textit{E-mail address}: \texttt{\href{mailto:gbenedetti@mathi.uni-heidelberg.de}{gbenedetti@mathi.uni-heidelberg.de}} \par
  \addvspace{\medskipamount}
\noindent\textsc{Seoul National University, Department of Mathematical Sciences, Research Institute in Mathematics, Gwanak-Gu, 
	Seoul 08826, South Korea} \par  
  \textit{E-mail address}: \texttt{\href{mailto:jungsoo.kang@snu.ac.kr}{jungsoo.kang@snu.ac.kr}} \par
}}

\title{On a systolic inequality for closed magnetic geodesics on surfaces}
\author{Gabriele Benedetti and Jungsoo Kang}
\begin{document}
\maketitle
\begin{abstract}
We apply a local systolic-diastolic inequality for contact forms and odd-symplectic forms on three-manifolds to bound the magnetic length of closed curves with prescribed geodesic curvature (also known as magnetic geodesics) on an oriented closed surface. Our results hold when the prescribed curvature is either close to a Zoll one or large enough.
\end{abstract}
\tableofcontents

\section{Introduction}
In this paper, we apply the systolic-diastolic inequality established in \cite{BK19a,BK19} for contact forms and odd-symplectic forms on three-manifolds, respectively, to the study of immersed closed curves with prescribed geodesic curvature on a connected oriented closed surface $(M,\mathfrak o_M)$ endowed with a Riemannian metric $g$. The Riemannian metric $g$ and the orientation $\mathfrak o_M$ yield a well-defined way of measuring angles in each tangent plane and an area form $\mu$ on $M$. If $c:I\to M$ is a smooth curve parametrised by arc-length on some interval $I$, we define its geodesic curvature $\kappa_c:I\to\R$ to be the unique function satisfying the relation
\[\label{e:kappac}
\nabla_{\dot c}\dot c=\kappa_c\dot c^\perp,
\]
where $\nabla$ is the Levi-Civita connection, and $\dot c^\perp$ is the unit vector with the property that the angle from $\dot{c}$ to $\dot{c}^\perp$ is $\tfrac\pi2$.

Let $f:M\to\R$ be a smooth function. A curve $c:\R\to M$ is said to be a \textbf{magnetic geodesic}, or an \textbf{$f$-magnetic geodesics} when we want to mention the function $f$ explicitly, if it is parametrised by arc-length and satisfies the equation
\begin{equation}\label{e:kappa}
\kappa_{c}(t)=-f(c(t)),\qquad\forall\, t\in \R.
\end{equation}
The magnetic geodesics of $f$ and of $-f$ are in one-to-one correspondence through time reversal. This means that $t\mapsto c(t)$ is an $f$-magnetic geodesic if and only if $t\mapsto c(-t)$ is a $-f$-magnetic geodesic. The study of periodic solutions of \eqref{e:kappa}, which we refer to as closed $f$-magnetic geodesics, is by now a problem with a rich history and we refer the reader to \cite{Tai92,Gin96b,CMP04,Ben16b,AB16} and the references therein for an account of the most remarkable developments and a generalization to higher dimensional manifolds $M$.\\ [-2ex]

A crucial ingredient in our work will be to use that the tangent lifts $(c,\dot c)$ of $f$-magnetic geodesics are the integral curves of a vector field $X_f$ defined on the unit tangent bundle $\ta^1M$, whose elements are the tangent vectors of unit norm. The foot-point projection 
\[
\mathfrak p_\infty:\ta^1M\to M
\] 
is an orientable $S^1$-bundle, whose fibres we orient by the $\mathfrak o_M$-{\it negative} direction. If $e\in H^2_\dR(M)$ is {\it minus} the real Euler class of $\mathfrak p_\infty$, then
\begin{equation}\label{e:state}
\langle e,[M]\rangle=\chi(M),
\end{equation}
where $\chi(M)$ is the Euler characteristic of $M$. We write $\mathfrak h_\infty\in[S^1,\ta^1M]$ for the free-homotopy class of $\mathfrak p_\infty$-fibres. We put the orientation $\mathfrak o_{\ta^1 M}$ on $\ta^1M$ combining the orientation $\mathfrak o_M$ on $M$ with the orientation of the $\mathfrak p_\infty$-fibres given above.

\begin{dfn}\label{def:Zoll_function}
We say that a function $f:M\to \R$ is {\bf Zoll} with respect to a given metric $g$ if there exists an oriented $S^1$-bundle
\[
\mathfrak p_f:\ta^1M\to M_f
\]
such that the integral curves of $X_f$ are fibres of $\mathfrak p_f$. We write $e_f$ for minus the Euler class of $\mathfrak p_f$, and $\mathfrak h_f$ for the free-homotopy class of the $\mathfrak p_f$-fibres. 
\end{dfn}
If we take $f\equiv0$, we recover the notion of Zoll Riemannian metric and $M$ must be the two-sphere. We refer the reader to \cite{Bes78} for a thorough discussion of such metrics. A classical example of a Zoll function $f_*:M\to\R$ can be given when $g=g_*$ is a metric of constant Gaussian curvature $K_*$. We take $f_*$ to be any constant function satisfying
\begin{equation}\label{e:f*K}
f_*^2+K_*>0.
\end{equation}
If $c$ is a prime closed magnetic geodesic, its lift $\widetilde c$ to the universal cover $\widetilde M$ of $M$ parametrises the boundary of a geodesic ball of radius
\[
R=\begin{cases}
\frac{1}{\sqrt{K_*}}\arctan\Big (\frac{\sqrt{K_*}}{|f_*|}\Big),&\text{if }K_*>0;\\
\frac{1}{|f_*|},&\text{if }K_*=0;\\
\frac{1}{\sqrt{-K_*}}\mathrm{arctanh}\Big(\frac{\sqrt{-K_*}}{|f_*|}\Big),&\text{if }K_*<0. 
\end{cases}
\]
According to our sign convention, the curve rotates clockwise, if $f_*>0$. Therefore, all $f_*$-magnetic geodesics are closed, and actually $f_*$ is Zoll. Here the map $\mathfrak p_{f_*}:\ta^1M\to M_{f_*}=M$ in Definition \ref{def:Zoll_function} associates to a tangent vector the projection on $M$ of the center of the corresponding ball in $\widetilde M$. In general, it is unknown whether every Riemannian metric admits a Zoll function.\\[-2ex]
\begin{rmk}\label{r:noorbits}
If we take $f_*=0$ and $K_*=0$ for the two-torus or $f_*^2+K_*<0$ for higher genus surfaces, all the closed magnetic geodesics are not contractible. If we take $f_*^2+K_*=0$ on higher genus surfaces, then there are no closed magnetic geodesics at all.
\end{rmk}
Let us go back to the case of an arbitrary function $f:M\to\R$ and attach two quantities to it. The former is the \textbf{average} of $f$:
\begin{equation*}
f_\avg:=\frac{1}{\area(M)}\int_M f\mu\,,\qquad\qquad\area(M):=\int_M\mu\,.
\end{equation*}
The latter is the \textbf{average curvature} of $f$, which generalises the left-hand side of \eqref{e:f*K}:
\begin{equation*}
K_f:=(f_\avg)^2+\frac{2\pi\chi(M)}{\area(M)}.
\end{equation*}
Indeed, by the Gauss-Bonnet theorem,
\[
\frac{2\pi\chi(M)}{\area(M)}=\frac{1}{\area(M)}\int_M K\mu,
\]
where $K$ is the Gaussian curvature of $g$. The average curvature of $f$ is always positive for the two-sphere $M=S^2$. For the two-torus $M=\T^2$, it is always non-negative and equality holds exactly when $f_\avg=0$.

In the next proposition, we collect the first properties of Zoll functions and their closed magnetic geodesics.
\begin{prp}\label{p:magzollintro}
If $f:M\to \R$ is a Zoll function, the following statements hold:
\begin{enumerate}[(a)]
\item The surface $M_f$ is diffeomorphic to $M$ and there is a path of oriented $S^1$-bundles $\{\mathfrak p_r \}_{r\in[0,1]}$ with total space $\ta^1 M$ and a sign $\epsilon(f)\in\{-1,+1\}$ such that
\[
\mathfrak p_0=\epsilon(f)\mathfrak p_\infty,\qquad \mathfrak p_1=\mathfrak p_f
\]
where $-\mathfrak p_\infty$ is the bundle $\mathfrak p_\infty$ with opposite orientation. In particular,
\begin{equation*}
\mathfrak h_f=\epsilon(f)\mathfrak h_\infty,\qquad\langle e_f,[M_f]\rangle=\chi(M).
\end{equation*}
\item If $M$ is not the two-sphere, then $f_\avg\neq0$ and $\epsilon(f)=\mathrm{sign}(f_\avg)$.
\item The average curvature is positive, namely $K_f>0$.
\end{enumerate}
\end{prp}
\begin{rmk}
If $M$ is the two-sphere, the oriented bundle $\mathfrak p_\infty$ is homotopic to its opposite $-\mathfrak p_\infty$. Therefore, both $\epsilon(f)=-1$ and $\epsilon(f)=+1$ are good in this case.
\end{rmk}
In view of this proposition, given a function $f:M\to\R$ which is not necessarily Zoll, we are motivated to look for closed $f$-magnetic geodesics whose tangent lift belongs either to the free-homotopy class $\mathfrak h_\infty$ or to $-\mathfrak h_\infty$. Up to substituting $f$ with $-f$, we focus on $f$-magnetic geodesics in the former class of curves, namely in the set
\begin{equation}
\Lambda(M;\mathfrak h_\infty):=\Big\{ c:\R/T\Z\to M\ \text{for some }T>0\ \Big|\ |\dot c|\equiv 1,\  [(c,\dot c)]=\mathfrak h_\infty\Big\}.
\end{equation}
Lemma \ref{l:turn} and Lemma \ref{l:homC} explain in more detail which curves belong to $\Lambda(M; \mathfrak h_\infty)$. We denote by $\Lambda(f;\mathfrak h_\infty)$ the subset of closed $f$-magnetic geodesics in $\Lambda(M; \mathfrak h_\infty)$. 
%Finally, motivated by \eqref{e:Kf} and the examples in Remark \ref{r:noorbits}, we also require that
%\begin{equation*}
%K_f>0\tag*{(P2)},
%\end{equation*}
%which, assuming $(P1)$, is non void only if $M$ is different from the two-sphere and the two-torus. 

We now describe a function $\ell_{f}:\Lambda(M; \mathfrak h_\infty)\to \R$ called \textbf{$f$-magnetic length functional}, whose critical set is exactly $\Lambda(f;\mathfrak h_\infty)$. To this purpose, let $c\in\Lambda(M;\mathfrak h_\infty)$. There exists a cylinder $\Gamma:[0,1]\times S^1\to\ta^1M$ such that $\Gamma(0,\cdot)$ is an oriented $\mathfrak p_\infty$-fibre and $\Gamma(1,\cdot)$ coincides with $(c,\dot c)$, up to reparametrisation. We regard the projection $\mathfrak p_\infty\circ\Gamma$ as a disc $C:D^2\to M$ bounding $c$. Any disc arising in this way is called an {\bf admissible capping disc} for $c$. We set
\[
\ell_{f}:\Lambda(M; \mathfrak h_\infty)\to \R,\qquad \ell_{f}(c) := \ell(c)+ \int_{D^2}C^*(f\mu),
\]
where $\ell(c)$ is the Riemannian length of $c$, and $C$ is an admissible capping disc for $c$. As will be shown in Section \ref{ss:unit}, the value of $\ell(c)$ is independent of the choice of $C$. The systolic-diastolic inequality will give bounds for the quantities
\[
\ell_{\min}(f):=\inf_{\substack{c\,\in\,\Lambda(f;\mathfrak h_\infty)\\ c\; \text{prime}}}\ell_{f}(c),\qquad \ell_{\max}(f):=\sup_{\substack{c\,\in\,\Lambda(f;\mathfrak h_\infty)\\ c\; \text{prime}}}\ell_{f}(c),
\]
in terms of the \textbf{average length} of $f$ which is defined by
\begin{equation}\label{e:al}
\bar\ell(f):=\frac{2\pi}{f_\avg+\sqrt{K_f}}.
\end{equation}
\begin{rmk}\label{r:ellbar}
For $M=S^2$, we automatically have $\bar{\ell}(f)>0$. If $M=\T^2$, then $\bar{\ell}(f)$ is a real number if and only if $f_\avg> 0$ and in this case $\bar{\ell}(f)=\pi/f_\avg>0$. If $M$ has higher genus, $\bar{\ell}(f)$ is a real number, if and only if $K_f\geq 0$. In this case, $f_\avg$ and $\bar{\ell}(f)$ are both non-zero and have the same sign.
\end{rmk}
\begin{dfn}\label{d:magsys}
We say that $f:M\to\R$ satisfies the \textbf{magnetic systolic-diastolic inequality} if $\bar{\ell}(f)$ is a well-defined real number and
\[
\ell_{\min}(f)\leq \bar\ell(f)\leq \ell_{\max}(f),
\]
with any of the two equalities holding if and only if $f$ is a Zoll function, whose magnetic geodesics lie in $\Lambda(M;\mathfrak h_\infty)$.
\end{dfn}
\begin{rmk}
According to Proposition \ref{p:magzollintro}, if $M\neq S^2$, the magnetic geodesics of a Zoll function $f$ lie in $\Lambda(M;\mathfrak h_\infty)$ if and only if $f_\avg>0$. In this case $K_f$ and $\bar{\ell}(f)$ are also positive.
\end{rmk}
\begin{rmk}\label{rmk:minus_homotopy_class}
One could define \textit{mutatis mutandis} the analogous space $\Lambda(M;-\mathfrak h_\infty)$ and give a corresponding variational principle and a systolic-diastolic inequality for closed $f$-magnetic geodesics contained therein. The only difference is that one has to substitute $\bar{\ell}(f)$ with $\bar{\ell}'(f):=2\pi(-f_\avg+\sqrt{K_f})^{-1}$.
\end{rmk}

We prove the inequality in two cases. First, we show it for functions close to a Zoll one.
\begin{thm}\label{t:mag2}
Let $M$ be a connected oriented closed surface endowed with a Riemannian metric, and let $f_*:M\to\R$ be a Zoll function, whose magnetic geodesics lie in $\Lambda(f;\mathfrak h_\infty)$. Then, there exists a $C^2$-neighbourhood $\mathcal F$ of $f_*$ in the space of functions such that every $f$ in $\mathcal F$ satisfies the magnetic systolic-diastolic inequality.
\end{thm}
Next, we establish the magnetic systolic-diastolic inequality for positive functions with large average. To make this concept precise, we need a definition. 
\begin{dfn}\label{d:strong}
For every $k\in \N$ and every $f:M\to(0,\infty)$, we set
\begin{equation*}
\langle f\rangle_k:=\frac{\Vert f\Vert_{C^k}}{\min f}\in[1,\infty).
\end{equation*}
For a constant $C>0$, we say that $f:M\to(0,\infty)$ is \textbf{$C$-strong}, if there holds
\begin{equation*}
f_\avg>\Big(\langle f\rangle_3^4+\langle f\rangle_2^6\Big)e^{C\langle f\rangle_1^2}.
\end{equation*}
\end{dfn}
 
\begin{thm}\label{t:mag}
Let $M$ be a connected oriented closed surface endowed with a Riemannian metric $g$. There exists a constant $C_g>0$ with the property that, if $f:M\to\R$ is $C_g$-strong, then the function $f$ satisfies the magnetic systolic-diastolic inequality.
\end{thm}

\begin{rmk}
It is plausible that Theorem \ref{t:mag2}  still holds if we let the metric $g$ also vary. To be precise, if $f_*$ is Zoll with respect to a metric $g_*$, then there should exist a $C^3$-neighbourhood $\mathcal G$ of $g_*$ and a $C^2$-neighbourhood $\mathcal F$ of $f_*$ such that for every $(g,f)\in\mathcal G\times \mathcal F$, $f$ satisfies the magnetic systolic-diastolic inequality with respect to $g$. Actually, in the purely Riemannian case (namely, when $f=0$), the systolic-diastolic inequality holds true for metrics $g$ on $S^2$, whose curvature is suitably pinched, see \cite{ABHS17} and also \cite[Corollary 4]{ABHS15}. We also expect Theorem \ref{t:mag} to be true if we let $g$ vary in a $C^3$-bounded set.
\end{rmk}

For all positive real numbers $s$, we have
\begin{equation}\label{e:rescale}
(sf)_\avg=s(f_\avg),\qquad\langle sf\rangle_k=\langle f\rangle_k,\quad \forall\,k\in\N.
\end{equation}
Thus, Theorem \ref{t:mag} applies to large rescalings of any positive function.
\begin{cor}\label{c:mag}
Let $M$ be a connected oriented closed surface endowed with a Riemannian metric $g$. For every $f:M\to(0,\infty)$, there exists a positive number $s(g,f)>0$ such that if $s>s(g,f)$, then the function $sf$ satisfies the magnetic systolic-diastolic inequality.\qed
\end{cor}
Theorem \ref{t:mag2} and Theorem \ref{t:mag} are consequences of the contact systolic-diastolic inequality established in \cite{BK19a} when $M$ is different from the two-torus, as in this case the tangent lifts of magnetic geodesics are the trajectory of a Reeb flow on the unit tangent bundle, up to reparametrisation. If $M$ is the two-torus, its unit tangent bundle is trivial, and results in \cite{BK19a} are not applicable. Instead, in this case, the theorems follow from the systolic-diastolic inequality for odd-symplectic forms explored in \cite{BK19}.
%\subsection*{Structure of the paper} \marginpar{to be revised}
%\begin{itemize}
%\item In Section \ref{s:applications}, we give the relevant definitions to the formulation of the magnetic systolic-diastolic inequality, and compute the volume and the action in this setting. Lemma \ref{l:magzollintro} is a consequence of Lemma \ref{l:magzollintro} and Corollary \ref{c:sign}. \vspace{-4pt}
%\item In the two subsections making up Section \ref{sec:magsys}, we establish the magnetic systolic-diastolic inequality for the cases considered in Theorem \ref{t:mag2} and \ref{t:mag}.
%\end{itemize}
\bigskip

\noindent\textbf{Acknowledgements.} This work is part of a project in the Collaborative Research Center \textit{TRR 191 - Symplectic Structures in Geometry, Algebra and Dynamics} funded by the DFG. It was initiated when the authors worked together at the University of M\"unster and partially carried out while J.K.~was affiliated with the Ruhr-University Bochum. We thank Peter Albers, Kai Zehmisch, and the University of M\"unster for having provided an inspiring academic environment. We are indebted to Stefan Suhr for the proof of $\epsilon(f)=\mathrm{sign}(f_\avg)$ for $M\neq \T^2$ in Proposition \ref{p:magzollintro}. G.B.~would like to express his gratitude to Hans-Bert Rademacher and the whole Differential Geometry group at the University of Leipzig. J.K.~is supported by Samsung Science and Technology Foundation under Project Number SSTF-BA1801-01.

\section{Preliminaries}\label{s:applications}

\subsection{The unit tangent bundle}\label{ss:unit}

As mentioned in the introduction, $f$-magnetic geodesics $c$ yield trajectories $(c,\dot c)$ of a flow $\Phi^{X_f}$ on the unit tangent bundle $\ta^1M$. The generating vector field $X_f$ can be explicitly written as
\begin{equation*}\label{eq:X_f}
X_f=X+\frac{1}{2\pi}(f\circ\mathfrak p_\infty)V,
\end{equation*}
where $X$ is the geodesic vector field of $g$, and $V$ is the vector field whose flow rotates the fibres of the map $\mathfrak p_\infty$ in the $\mathfrak o_M$-\textit{negative} direction with constant angular speed $\tfrac{1}{2\pi}$. Thus, the vector field $V$ generates a free $S^1$-action on $\ta^1M$ (our convention is $S^1=\R/\Z$) and we denote by \label{dfn:hinfty}${\mathfrak h_\infty}\in[S^1,\ta^1M]$ the free-homotopy class of the orbits of $V$, namely of the oriented $\mathfrak p_\infty$-fibres. The Levi-Civita one-form $\eta\in\Omega^1(\ta^1M)$ is the connection for $\mathfrak p_\infty$ satisfying
\begin{equation}
\label{eq:Levi_Civita}
\eta(V)=1,\qquad \di\eta=\frac{1}{2\pi}\mathfrak p_\infty^*(K\mu),
\end{equation}
where $K$ is the Gaussian curvature of $g$. This implies that $e=\tfrac{1}{2\pi}[K\mu]$, where $e$ is minus the real Euler class of $\mathfrak p_\infty$.
   
Let $\alpha_{\mathrm{can}}$ be the canonical one-form on $\ta^1M$ given by 
\[
(\alpha_{\mathrm{can}})_v\cdot Y:=g(v,\di_v\mathfrak p_\infty\cdot Y),\qquad \forall\,Y\in \ta_v(\ta^1M).
\] 
It is a contact form and its Reeb vector field coincides with the geodesic vector field $X$. There holds (see \cite[V.2.5 and (5.2.12)]{Ber65})
\begin{equation}\label{eq:alpha_can}
\alpha_{\mathrm{can}}\wedge\di\alpha_{\mathrm{can}}=2\pi\eta\wedge\mathfrak p_\infty^*\mu
\end{equation}
so that $\alpha_{\mathrm{can}}\wedge\di\alpha_{\mathrm{can}}$ is a positive form with respect to $\mathfrak o_{\ta^1 M}=\mathfrak o_V\oplus\mathfrak o_M$.

\begin{dfn}\label{d:odd}
We call a two-form $\Omega$ on $\ta^1M$ {\bf odd-symplectic} if it is closed and maximally non-degenerate. An odd-symplectic form is called {\bf Zoll} if there exists an oriented  $S^1$-bundle $\mathfrak p_\Omega:\ta^1M\to M_\Omega$ such that the oriented leaves of the distribution $\ker\Omega$ are fibres of $\mathfrak p_\Omega$. In this case, $\Omega$ descends to a symplectic form $\om$ on $M_\Omega$, i.e.~$\mathfrak p_\Omega^*\om=\Omega$. We endow $M_\Omega$ with the orientation induced by $\ta^1M$ and $\mathfrak p_\Omega$, or, equivalently, the orientation given by $\omega$.
\end{dfn}
\begin{rmk}
The two-form 
\[\label{eq:Omega_infty}
\Omega_\infty:=\mathfrak p_\infty^*\mu
\]
is an example of a Zoll odd-symplectic form. Its associated oriented bundle is $\mathfrak p_\infty$.
\end{rmk}

The two-form 
\[
\Omega_f:=\di\alpha_{\mathrm{can}}+\mathfrak p_\infty^*(f\mu),
\]
is odd-symplectic, and the vector field $X_f$ is a nowhere vanishing section of the characteristic distribution $\ker\Omega_f$. Indeed, from the equations above, we have $\Omega_f=\iota_{X_f}(\alpha_{\mathrm{can}}\wedge\di\alpha_{\mathrm{can}})$. This also shows that
\begin{equation}\label{e:orientation}
\mathfrak o_{\ta^1M}=\mathfrak o_{\Omega_f}\oplus \mathfrak o_{X_f},
\end{equation}
where $\mathfrak o_{\Omega_f}$ is the co-orientation of the characteristic distribution of $\Omega_f$. We readily see that $f$ is Zoll in the sense of Definition \ref{def:Zoll_function} if and only if $\Omega_f$ is Zoll in the sense of Definition \ref{d:odd}.

To determine the cohomology class $[\Omega_f]=f_\avg[\Omega_{\infty}]\in H^2_\dR(\ta^1M)$, we observe that the map $\mathfrak p_\infty^*:H^2_\dR(M)\to H^2_\dR(\ta^1M)$ is zero if $M\neq\T^2$ and is injective if $M=\T^2$. This follows from the Gysin sequence
\begin{equation*}
H^0_\dR(M)\stackrel{\cup e}{\longrightarrow}H^2_\dR(M)\stackrel{\mathfrak p_\infty^*}{\longrightarrow} H^2_\dR(\ta^1M)
\end{equation*}
and \eqref{e:state}. Therefore, $\Omega_{\infty}$ is exact if and only if $M\neq\T^2$.

Let us now write the $f$-magnetic length of some $c\in\Lambda(M;\mathfrak h_\infty)$ in term of $\Omega_f$. The Riemannian length of $c$ can be expressed as
\begin{equation*}
\ell(c)=\int_{\R/T\Z}(c,\dot c)^*\alpha_{\mathrm{can}}.
\end{equation*}
If $\Gamma:[0,1]\x S^1\to\ta^1M$ is the cylinder lifting an admissible capping disc $C$, we have
\begin{equation}\label{eq:length_Gamma}
\ell_f(c)=\ell(c)+\int_{D^2}C^*(f\mu)=\int_{[0,1]\x S^1}\Gamma^*\Omega_f,
\end{equation}
due to $\int_0^1\Gamma(0,\cdot)^*\alpha_{\mathrm{can}}=0$ and Stokes' Theorem. From this formula, we deduce that the value of $\ell_f(c)$ does not depend on the choice of the admissible disc $C$. Let $C'$ be another admissible capping disc for $c$, which is the projection of another cylinder $\Gamma'$ in $\ta^1M$ such that $\Gamma'(0,\cdot)$ is an oriented $\mathfrak p_\infty$-fibre and $\Gamma'(1,\cdot)$ coincides with $(c,\dot c)$, up to reparametrisation. The cylinder $\Gamma''$ obtained concatenating $s\mapsto \Gamma'(s,\cdot)$ with the reversed cylinder $s\mapsto \Gamma(1-s,\cdot)$ projects to a sphere $\sigma:S^2\to M$. The value of $\ell_f(c)$ obtained using $C$ is the same as the one obtained using $C'$ if and only if the following integral vanishes
\[
\int_{[0,1]\x S^1}(\Gamma'')^*\Omega_f=\int_{[0,1]\x S^1}(\Gamma'')^*\big(\mathfrak p_\infty^*(f\mu)\big)=\langle [f\mu],[\sigma]\rangle.
\]
Therefore, it is enough to show that $[\sigma]=0$. If $M\neq S^2$, this is clear since $\pi_2(M)$ vanishes. If $M=S^2$, then $[\sigma]=0$ if and only if $\langle e,[\sigma]\rangle=0$. By \eqref{eq:Levi_Civita} and Stokes' Theorem, we get
\[
\langle e,[\sigma]\rangle=\langle [\tfrac{1}{2\pi}K\mu],[\sigma]\rangle=\int_{[0,1]\times S^1}(\Gamma'')^*\eta=\int_{S^1}(\gamma_0')^*\eta-\int_{S^1}\gamma_0^*\eta=1-1=0.
\]

\subsection{Proof of Proposition \ref{p:magzollintro}}
\subsubsection*{Case $M=\T^2$}
Since $\mathfrak p_\infty$ is a trivial bundle, \cite[Proposition 1.9]{BK19} implies that $e_f=0$ and that $\Omega_f$ is not exact. Moreover, by \cite[Lemma 4.5]{BK19}, if $\mathfrak p:\ta^1\T^2\to M_{\mathfrak p}$ is an oriented $S^1$-bundle over a closed surface $M_{\mathfrak p}$, for any $c\in H^2_{\dR}(M_{\mathfrak p})$ and $\pt\in M_{\mathfrak p}$ there holds
\begin{equation*}
\mathrm{PD}(\mathfrak p^*c)=\langle c,[M_{\mathfrak p}]\rangle\cdot[\mathfrak p^{-1}(\pt)]\in H_1(\ta^1\T^2;\R)
\end{equation*}
where $\mathrm{PD}$ denotes Poincar\'e duality. Applying this identity for $\mathfrak p=\mathfrak p_\infty$, we get
\begin{equation*}
\mathrm{PD}([\Omega_f])=\mathrm{PD}([\mathfrak p_\infty^*(f\mu)])=\area(\T^2)\cdot f_\avg\cdot[\mathfrak p_\infty^{-1}(\mathrm{pt})]\in H_1(\ta^1\T^2;\R),
\end{equation*}
which implies $f_\avg\neq0$ and, as a consequence, $K_f>0$. Applying again the identity for $\mathfrak p=\mathfrak p_f$, we deduce
\[
\mathrm{PD}([\Omega_f])=\langle [\omega_f],[\T^2_f]\rangle\cdot [\mathfrak p_f^{-1}(\pt)],
\]
where $\langle[\om_f],[\T^2_f]\rangle>0$, as $\om_f$ is a positive symplectic form. By comparing the two formulae for $\mathrm{PD}([\Omega_f])$, we derive $[\mathfrak p_f^{-1}(\pt)]=\mathrm{sign}(f_\avg)\cdot [\mathfrak p_\infty^{-1}(\mathrm{pt})]$, as the homology classes of the fibres of $\mathfrak p_\infty$ and of $\mathfrak p_f$ are both primitive. Since $\ta^1\T^2$ is diffeomorphic to the three-torus, we have an isomorphism between the set of free-homotopy classes and the set of first homology classes, so that $\mathfrak h_f=\mathrm{sign}(f_\avg)\mathfrak h_\infty$ holds, as well. Finally, as $[\Omega_f]=f_\avg[\Omega_\infty]$, the existence of a path $\{\mathfrak p_r \}$ connecting $\mathrm{sign}(f_\avg)\mathfrak p_\infty$ to $\mathfrak p_f$ follows from \cite[Proposition 1.9, Remark 1.10]{BK19}.
\qed

\subsubsection*{Case $M\neq\T^2$}
The existence of a path $\{\mathfrak p_r\}$ connecting $\pm\mathfrak p_\infty$ to $\mathfrak p_f$ is a consequence of \cite[Proposition 1.2]{BK19a}. It implies at once that $\mathfrak h_f=\pm\mathfrak h_\infty$, and by continuity also that $\langle e_f,[M_f]\rangle=\chi(M)$, since \eqref{e:state} holds. Notice indeed that the Euler number of $\mathfrak p_\infty$ is also equal to $\chi(M)$, since the Euler class of $-\mathfrak p_\infty$ is minus the Euler class of $\mathfrak p_\infty$ and $-\mathfrak p_\infty$ induces the opposite orientation on $M$, so the two minus signs cancel out when computing the Euler number. When $M=S^2$, there is nothing else to prove, so let us assume for the rest of the proof that $\chi(M)<0$. In this case, the inequality $K_f>0$ is proven in Corollary \ref{c:sign} and we are left to establish (b). We will show, namely, that $f_\avg>0$, provided the magnetic geodesics of $f$ lie in $\Lambda(M;\mathfrak h_\infty)$. We assume by contradiction that $f_\avg<0$. Thanks to Remark \ref{r:ellbar}, this is equivalent to assuming that $\bar{\ell}(f)<0$. By Theorem \ref{t:mag2}, we have $\ell_f(c)=\bar{\ell}(f)$ for every prime closed $f$-magnetic geodesic $c$. Notice that we are allowed to use Theorem \ref{t:mag2}, since, when $M\neq \T^2$, such a result depends only on part (a) and (c) of Proposition \ref{p:magzollintro}. Let $L_0:\ta M\to\R$ be the energy density $L_0(q,v)=\tfrac12g_q(v,v)$ and let $\Lambda_0(M)$ be the set of contractible loops on $M$ with arbitrary period. We define the Lagrangian free-period action functional $S_k^{L_0}:\Lambda_0(M)\to\R$ with parameter $k\in\R$ by
\[
S_k^{L_0}(c):=\int_0^T\Big[L_0+k\Big](c(t),\dot c(t))\di t+\int_{D^2}C^*(f\mu),\quad\forall\,c\in \Lambda_0(M),
\]
where $C:D^2\to M$ is a capping disc for $c\in \Lambda_0(M)$. The definition does not depend on $C$ since $\pi_2(M)=0$. Moreover, observe that if $c^m:\R/mT\Z\to M$ is the $m$-th iteration of $c$, we have
\begin{equation}\label{e:iteration}
S_k^{L_0}(c^m)=mS_k^{L_0}(c).
\end{equation}

Let $\mathcal L^{L_0}:\ta M\to \ta^*M$ be the Legendre transform associated with the Lagrangian $L_0$ and let $H_0:\ta^*M\to\R$ be the kinetic energy function with respect to the dual metric. The function $H_0$ is Legendre dual to $L_0$, namely
\[
L_0(q,v)=\mathcal L^{L_0}(q,v)\cdot v-H_0\big(\mathcal L^{L_0}(q,v)\big),\quad \forall\,(q,v)\in\ta M.
\]
Let $\widehat\Omega_f$ be the twisted symplectic form on $\ta^*M$, which is defined by
\[
\widehat{\Omega}_f:=\di\widehat{\alpha}_\can+\widehat{\mathfrak p}_\infty^*(f\mu),
\]
where $\widehat{\alpha}_\can$ is the canonical one-form on $\ta^*M$ and $\widehat{\mathfrak p}_\infty:\ta^*M\to M$ is the foot-point projection. We also define the Hamiltonian free-period action functional $A_k^{H_0}:\Lambda_0(\ta^*M)\to\R$ on the set of contractible loops in $\ta^*M$ by
\[
A_k^{H_0}(q,p):=\int_{D^2}(Q,P)^*\widehat{\Omega}_f+\int_0^T\Big[k-H_0\Big](q(t),p(t))\di t,\quad\forall\,(q,p)\in\Lambda_0(\ta^*M),
\]
where $(Q,P):D^2\to\ta^*M$ is any capping disc for $(q,p)\in\Lambda_0(\ta^*M)$. Since $\pi_2(\ta^*M)=0$, the definition does not depend on $(Q,P)$. For all $c\in\Lambda_0(M)$, there holds
\[
S_k^{L_0}(c)=A_k^{H_0}\big(\mathcal L^{L_0}(c,\dot c)\big).
\]
It is a classical result that the Hamiltonian flow lines of $H_0$ with respect to $\widehat{\Omega}_f$ on the energy hypersurface $\{H_0=\tfrac12\}$ are exactly the curves $\mathcal L^{L_0}(c,\dot c)$, where $c$ is an $f$-magnetic geodesic. Therefore, if $c$ is a prime closed $f$-magnetic geodesic, we have
\begin{equation}\label{e:SA}
0>\ell_f(c)=S_{\frac12}^{L_0}(c)=A_{\frac12}^{H_0}(\mathcal L^{L_0}(c,\dot c))=\int_{D^2}(Q_c^{L_0},P_c^{L_0})^*\widehat{\Omega}_f,
\end{equation}
where $(Q_c^{L_0},P_c^{L_0})$ is a capping disc for $\mathcal L^{L_0}(c,\dot c)$.

The two-form $\Omega_f$ on $\ta^1M$ is exact since it is Zoll and $\mathfrak p_f$ is non-trivial. Thus, there exists a contact form $\lambda_f$ on $\ta^1M$ such that $\di\lambda_f=\Omega_f$. This implies that $\{H_0=1/2\}$ is a stable hypersurface inside $(\ta^*M,\widehat{\Omega}_f)$ in the sense of \cite[Section 4.3]{HZ11} and \cite[Section 2]{CM05}. By \cite[Lemma 2.1]{MP10}, there exist $k_*\in\R$, an open interval $I$ containing $k_*$, a function $h:I\to(0,\infty)$, a Tonelli Hamiltonian $H:\ta^*M\to\R$ and a diffeomorphism $\Psi:\{H=k_*\}\times I\to \{H\in I\}$ such that the following identities hold
\begin{align*}
(i)&\ \ \{H_0\leq 1/2 \}=\{H\leq k_*\},\\
(ii)&\ \ H\circ\Psi(p,r)=r,\quad \forall\,(p,r)\in \{H=k_*\}\times I,\\
(iii)&\ \ \Psi\Big(\Phi_H^{h(r)t}(p),r\Big)=\Phi_H^t(\Psi(p,r)),\quad\forall\,t\in\R,\ (p,r)\in \{H=k_*\}\times I,
\end{align*}
where $\Phi_H$ is the Hamiltonian flow of $H$. Actually, the result in \cite{MP10} is stated for Hamiltonians on the standard cotangent bundle $(T^*M,\widehat{\Omega}_0)$ but a careful inspection of the 
proof reveals that the statement holds also on the twisted cotangent bundle $(T^*M,\widehat{\Omega}_f)$. Finally, let $L:\ta M\to\R$ be the Legendre dual of $H$. From (i), it follows that $\Phi_H$ and $\Phi_{H_0}$ have the same oriented orbits on $\{H_0=1/2 \}=\{H=k_*\}$. Therefore, if $c_1$ is a prime periodic solution of the Euler-Lagrange flow of $L$ with energy $k_*$, we have that
\begin{equation}\label{e:D2}
\int_{D^2}(Q_c^{L_0},P_c^{L_0})^*\widehat{\Omega}_f=\int_{D^2}(Q_{c_1}^{L},P_{c_1}^{L})^*\widehat{\Omega}_f,
\end{equation}
where $(Q_{c_1}^{L},P_{c_1}^{L})$ is a capping disc for $\mathcal L^L(c_1,\dot{c_1})$. The left-hand side of \eqref{e:D2} is negative by \eqref{e:SA}. Moreover, using the last two passages in \eqref{e:SA} backwards with $k_*$, $c_1$, $H$ and $L$ instead of $\frac12$, $c$, $H_0$ and $L_0$, we see that the right-hand side of \eqref{e:D2} is equal to $S^L_{k_*}(c_1)$. Summing up, we have shown that $S^L_{k_*}(c_1)<0$. By (ii),(iii) and \eqref{e:iteration}, we see that, up to shrinking the interval $I$, we can assume that all periodic solutions (prime and iterated) of the Euler-Lagrange flow of $L$ with energy $k\in I$ have negative action $S^L_k$. However, since $\mathfrak p_\infty:\{H=k_*\}\to M$ is an $S^1$-bundle and there exist contractible curves with negative $S^L_k$ action for all $k\in I$, we conclude that $I\subset(c_*(L),c_u(L))$, where $c_*(L)=-\min_{q\in M} L(q,0)$ and $c_u(L)\in\R$ is the Ma\~n\'e critical value of the universal cover. However, by \cite[Theorem 1.1(2)]{M10} or \cite[Theorem 1.3]{AB16}, there exists for almost every $k\in(c_*(L),c_u(L))$ a period orbit of the Euler-Lagrange flow of $L$ with energy $k$ and positive $S^L_k$-action. This contradiction shows that $\bar{\ell}(f)>0$ and finishes the proof.\qed  
%\begin{rmk}\label{r:omfpos}
%If $f_*:M\to\R$ is a Zoll function, we have 
%\[
%\langle [\om_{f_*}],M_{f_*}\rangle>0.
%\]
%Moreover, by \eqref{e:orientation}, the lifts $(c,\dot c)$ of $f_*$-magnetic geodesics parametrise the $\mathfrak p_{f_*}$-fibres with the positive orientation. 
%If in addition $M=\T^2$, we see in the proof of Lemma \ref{l:magzollintro} that
%\[
%\langle [\om_{f_*},[\T_{f_*}]\rangle=\area(\T^2)|(f_*)_\avg|.
%\]
%\end{rmk}

\subsection{The space of curves $\Lambda(M;\mathfrak h_\infty)$}
In this subsection, we will study the set $\Lambda(M; \mathfrak h_\infty)$ in more detail. We start with a characterisation of this space by means of the turning number of an immersed curve $b:\R/T\Z\to\R^2$ that is the winding number of its velocity curve $\dot b:\R/T\Z\to\R^2$ with respect to $0\in\R^2$.
\begin{lem}\label{l:turn}
Let $c$ be an immersed closed curve in $M$ that is contractible. The curve $c$ belongs to $\Lambda(M; \mathfrak h_\infty)$ if and only if the following condition holds.
\begin{description}
	\item[\normalfont \textit{$\bullet$ Case $M=S^2$.}]The turning number of $\psi\circ c$ is odd, where $\psi:S^2\setminus\{q\}\to\R^2$ is a diffeomorphism and $q\in M$ lies outside the support of $c$.
	\item[\normalfont \textit{$\bullet$ Case $M\neq S^2$.}] The turning number of $\widetilde c$ is equal to $-1$, where $\widetilde c:\R/T\Z\to \widetilde M\subset \R^2$ is a lift of $c$ to the universal cover of $M$. In this case, the curve $c$ is prime.\qed
\end{description}
\end{lem}
A somewhat more geometrical sufficient condition for a curve to be in $\Lambda(M; \mathfrak h_\infty)$ is given by the notion of Alexandrov embeddedness.
\begin{dfn}
A closed and arc-length parametrised curve $c$ in $M$ is called {\bf negatively Alexandrov embedded} if it admits a negatively immersed capping disc $C:D^2\to M$. 
\end{dfn}
\begin{rmk}\label{r:schoen}
By the Sch\"onflies Theorem, a closed curve $c$ in $M$ is negatively Alexandrov embedded, if:
\begin{itemize}
\item $M=S^2$ and $c$ is embedded;
\item $M\neq S^2$ and the lift $\widetilde c$ to the universal cover $\widetilde M$ bounds a compact region in the clock-wise direction.  
\end{itemize}
\end{rmk}
\begin{lem}\label{l:homC}
If a closed curve $c$ in $M$ is negatively Alexandrov embedded, $c\in\Lambda(M; \mathfrak h_\infty)$ and any of its immersed capping discs is admissible. In particular, the curves from Remark \ref{r:schoen} belong to $\Lambda(M; \mathfrak h_\infty)$.
\end{lem}
\begin{proof}
Let $C:D^2\to M$ be a negatively immersed capping disc for $c$. Then, we can define
\[
(0,1]\times S^1\ni(s,t)\longmapsto \left(C(se^{2\pi it}),\frac{\partial_t C(se^{2\pi it})}{|\partial_t C(se^{2\pi it})|}\right)\in \ta^1 M.
\]
Since $C$ is a local embedding around $0\in D^2$, this map extends to $s=0$ and yields a cylinder $\Gamma:[0,1]\times S^1\to\ta^1 M$ such that
\begin{align*}
(i)&\ \ \mathfrak p_\infty(\Gamma(s,t))=C(se^{2\pi it}),\quad \forall\,(s,t)\in [0,1]\times S^1,\\
(ii)&\ \ \Gamma(0,t)=\Phi^{V}_{a(t)}(z),\quad \Gamma(1,t)=\big(c(tT),\dot c(tT)\big),\quad\forall\, t\in S^1,
\end{align*}
for some orientation-preserving diffeomorphism $a:S^1\to S^1$ and element $z\in \ta^1 M$. This shows that $C$ is admissible.
\end{proof}

We finish this subsection by providing a partial answer to the following natural question. If all the $f$-magnetic geodesics are closed, is the function $f$ (or, equivalently, the odd-symplectic two-form $\Omega_f$) Zoll? We collect the result in a lemma, which is a magnetic counterpart of the Gromoll-Grove Theorem \cite{GG82}.

\begin{lem}\label{l:charzollmag}
Suppose that every $f$-magnetic geodesic is closed. The function $f$ is Zoll in the following two cases:
\begin{enumerate}[(i)]
\item There holds $M\neq S^2$ and all the prime geodesics lie in $\Lambda(M; \mathfrak h_\infty)$.
\item There holds $M=S^2$ and either all prime closed magnetic geodesics are embedded or the function $f$ is positive and all prime closed magnetic geodesics are negatively Alexandrov embedded.
\end{enumerate}
\end{lem}
\begin{proof}
A theorem of Epstein \cite{Eps72} yields an $S^1$-action $\Phi_t:\ta^1 M\to\ta^1M$, $t\in S^1$, whose orbits coincide with the tangent lifts of magnetic geodesics (up to reparametrisation) and such that the set
\[
N:=\big\{z\in \ta^1 M\ |\ \Phi_t(z)\neq z,\ \forall\, t\in S^1\setminus 0\big\}
\]
is non-empty. The lemma follows once we show that $N=\ta^1M$. The set $N$ is open, so that, by the connectedness of $\ta^1M$, we just have to prove that $N$ is also closed. Let $(z_m)\subset N$ be a sequence such that $z_m\to z\in\ta^1 M$. Let $(c_m)$ be the corresponding sequence of magnetic geodesics and $c$ the magnetic geodesic corresponding to $z$. Since $z_m\to z$, there exists $k\in\N^*$ such that $(c_m)$ converges in the $C^\infty$-topology to the $k$-th iteration of $c$. It suffices to show that $k=1$. This would give that $z\in N$, and hence, that $N$ is closed.

Let us suppose that $M\neq S^2$. The lifts $(\widetilde c_m)$ and $\widetilde c$ to $\widetilde M$ are such that $(\widetilde c_m)$ converges to the $k$-th iteration of $\widetilde c$. From Lemma \ref{l:turn}, we conclude that $k$-times the turning number of $\widetilde c$ is equal to $-1$, which forces $k=1$.

Let us suppose that $M=S^2$. If all prime closed magnetic geodesics are embedded, then all the curves $c_m$ are embedded. Since $S^2$ is an oriented surface, it follows that $c$ is also embedded, which forces $k=1$. If $f$ is everywhere positive and the curves $c_m$ are negatively Alexandrov embedded, then by \cite[Lemma 3.2]{MS12}, $c$ is also negatively Alexandrov embedded. From \cite[Lemma 3.1]{MS12}, it follows that $c$ is prime, i.e.~$k=1$. 
\end{proof}
\begin{rmk}
In the previous lemma, we need extra conditions when $M=S^2$ since there exists a sequence of prime Alexandrov embedded curves $(c_m)$ which converges in the $C^\infty$-topology to a curve $c$, which is not prime. In particular, the set $\{c\in\Lambda(S^2;\mathfrak h_\infty)\ |\ c\text{ is prime}\,\}$ is not closed in the $C^\infty$-topology. Furthermore, there are examples of positive magnetic functions on the two-sphere all of whose magnetic geodesics are closed but their lifts to the unit tangent bundle are the orbits of a non-free $S^1$-action \cite{Ben16}.
\end{rmk}

\subsection{Strong magnetic functions}\label{ss:strongdef}

When $f:M\to (0,\infty)$ is large, then $f$-magnetic geodesics stay close to the fibres of $\mathfrak p_\infty$. In this case, we expect $\Omega_f$ to approximate the Zoll form $\Omega_\infty=\mathfrak p_\infty^*\mu$. Using the notion of $C$-strong function given in Definition \ref{d:strong}, we make this observation precise in the next lemma. This result will be employed in Section \ref{ss:strong} to establish the magnetic systolic-diastolic inequality for $C$-strong functions.
\begin{lem}\label{l:normal}
Let $\mathcal U$ be a $C^2$-neighbourhood of $\Omega_\infty$ in the space of two-forms on $\ta^1M$. There exists a constant $C_{\mathcal U}>0$ with the following property: For every $C_{\mathcal U}$-strong $f:M\to (0,\infty)$, there is a diffeomorphism $\Psi:\ta^1 M\to\ta^1 M$ isotopic to the identity such that $\tfrac{1}{f_\avg}\Psi^*\Omega_f\in\mathcal U$.
\end{lem}
\begin{proof}
%Throughout the proof, we denote by $C_k>0$ a positive constant (that might change during the computations) depending solely on $g$ and a natural number $k$. 
We define
\begin{equation*}\label{eq:f_norm}
f_\norm:=\frac{f}{f_\avg}
\end{equation*}
and observe that there holds
\begin{equation*}
\min f_\norm\leq 1\leq \max f_\norm.
\end{equation*}
The two-form $(f_\norm-1)\mu$ is exact. By standard elliptic arguments (see for instance \cite[Chapter 10]{Nic07}), we can choose a primitive one-form $\zeta\in\Omega^1(M)$ of $(f_\norm-1)\mu$ such that
\begin{equation}\label{e:schau}
\Vert \zeta\Vert_{C^{k}}\leq C_k\Vert f_\norm-1\Vert_{C^{k}}\leq C_k\Vert f_\norm\Vert_{C^{k}},\qquad \forall\, k\in\N
\end{equation}
for some constant $C_k>0$ depending solely on $g$ and $k\in\N$.
For $s\in[0,1]$, let $\mu_s$ be the two-form given by $\mu_s:=q(f,s)\mu$, where $q(f,s):=sf_\norm+(1-s)$, and $Y_s$ be the time-dependent vector field defined through 
\begin{equation*}
\iota_{Y_s}\mu_s=-\zeta.
\end{equation*}
If $\psi:M\to M$ is the time-one map of $Y_s$, an application of Moser's trick yields
\begin{equation}\label{e:mosermag}
\psi^*(f_\norm\mu)=\mu.
\end{equation}
If $\sharp:\ta^*M\to \ta M$ is the metric duality and $\ast:\ta^*M\to\ta^*M$ the Hodge star operator, we can write $Y_s$ explicitly as 
\[
Y_s=\frac{\sharp\ast\zeta}{q(f,s)}.
\]
Since $\ast$ and $\sharp$ are smooth bundle maps, we have (possibly with bigger $C_k>0$)
\begin{equation}\label{e:ys}
\Vert Y_s\Vert_{C^{k}}\leq C_k\max_{s\in[0,1]}\Big\Vert \frac{\zeta}{q(f,s)}\Big\Vert_{C^{k}},\qquad\forall\, k\in\N.
\end{equation}

We claim that the following bound holds (possibly with bigger $C_k>0$):
\begin{equation}\label{e:fnormminfnorm}
\max_{s\in[0,1]}\Big\Vert \frac{\zeta}{q(f,s)}\Big\Vert_{C^{k}}\leq C_k\langle f_\norm \rangle _k^{k+1}=C_k\langle f\rangle_k^{k+1},\qquad \forall\, k\in\N.
\end{equation}
where the last equality is due to \eqref{e:rescale}. We prove the claim by induction and observe preliminarily that $q(f,s)\geq \min f_\norm$. For $k=0$, the estimate follows directly from \eqref{e:schau}. Suppose now that the estimate holds for all $k'\leq k-1$. Since
\[
\Big\Vert \frac{\zeta}{q(f,s)}\Big\Vert_{C^{k}}=\Big\Vert \frac{\zeta}{q(f,s)}\Big\Vert_{C^{k-1}}+\Big\Vert \nabla^{k}\frac{\zeta}{q(f,s)}\Big\Vert_{C^{0}},
\]
we just have to bound the second term. We apply the Leibniz rule to the $k$-th derivative of the product $q(f,s)\cdot \frac{\zeta}{q(f,s)}=\zeta$ and obtain
\begin{align*}
\nabla^{k}\Big(\frac{\zeta}{q(f,s)}\Big)=\frac{1}{q(f,s)}\bigg[\nabla^{k}\zeta-s\sum_{k'=0}^{k-1}\binom{k}{k'}\nabla^{k-k'}f_\norm\cdot\nabla^{k'}\frac{\zeta}{q(f,s)}\bigg],
\end{align*}
where we have used that $\nabla^{k-k'}q(f,s)=s\nabla^{k-k'}f_\norm$, since $k-k'\geq 1$. Consequently, we estimate using \eqref{e:schau} and \eqref{e:fnormminfnorm}
\begin{align*}
\Big\Vert \nabla^{k}\frac{\zeta}{q(f,s)}\Big\Vert_{C^{0}}&\leq \frac{1}{\min f_\norm}\bigg[C_{k}\Vert f_\norm\Vert_{C^{k}}+\sum_{k'=0}^{k-1}\binom{k}{k'}\Vert f_\norm\Vert_{C^{k}}C_{k-1}\langle f_\norm\rangle_{k-1}^{k}\bigg]\\
&\leq \frac{1}{\min f_\norm}C'_{k}\bigg(\Vert f_\norm\Vert_{C^{k}}+\frac{\Vert f_\norm\Vert_{C^{k}}^{k+1}}{(\min f_\norm)^{k}}\bigg)\\
&\leq C_{k}'\frac{\Vert f_\norm\Vert_{C^{k}}}{\min f_\norm}+C_{k}'\bigg(\frac{\Vert f_\norm\Vert_{C^{k}}}{\min f_\norm}\bigg)^{k+1}\\
&\leq (C_{k}'+1)\bigg(\frac{\Vert f_\norm\Vert_{C^{k}}}{\min f_\norm}\bigg)^{k+1}
\end{align*}
with some $C_k'>0$ depending only on $g$ and $k$. The claim is therefore established. 

Using the Levi-Civita connection for $\mathfrak p_\infty$, we lift $Y_s$ horizontally to $Z_s$ on $\ta^1 M$, so that $\di\mathfrak p_\infty(Z_s)=Y_s$. Since the lifting map $Y_s\mapsto Z_s$ is smooth and depends only on $g$, but not on $f$, there is a constant $C_k''>0$ depending on $k$ and $g$ such that
\begin{equation}\label{e:zsys}
\Vert Z_s\Vert_{C^{k}}\leq C_k''\Vert Y_s\Vert_{C^{k}},\qquad \forall\, k\in\N.
\end{equation}
The time-one map $\Psi:\ta^1M\to\ta^1M$ of $Z_s$ lifts the time-one map $\psi$ of $Y_s$, so that
\[
\Psi^*(\mathfrak p_\infty^*(f_\norm\mu))=\mathfrak p_\infty^*\mu,
\]
by \eqref{e:mosermag}. Putting together \eqref{e:ys}, \eqref{e:fnormminfnorm}, \eqref{e:zsys}, and Lemma \ref{l:diff2}, we get
\begin{equation*}
B_{2,2}\big(\Vert\di\Psi\Vert\big)\leq \Big(\langle f\rangle_3^4+\langle f\rangle_2^6\Big)e^{C_3\langle f\rangle_1^2},
\end{equation*}
for a (possibly bigger) constant $C_3>0$. Hence, using \eqref{eq:Ck_estimate} we estimate \begin{equation*}
\Vert\Psi^*(\di\alpha_{\mathrm{can}})\Vert_{C^2}\leq C_3''' B_{2,2}\big(\Vert\di\Psi\Vert\big)\Vert\di\alpha_{\mathrm{can}}\Vert_{C^2}\leq \Big(\langle f\rangle_3^4+\langle f\rangle_2^6\Big)e^{C_3\langle f\rangle_1^2},
\end{equation*}
where $C_3'''>0$ depends only on $g$ and where we take a bigger constant $C_3>0$ if necessary to incorporate $\Vert\di\alpha_\mathrm{can}\Vert_{C^2}$ and it is possible to bring the constant 
to the exponent since $\langle f\rangle_1^2\geq 1$. 

Let us suppose now that $f$ is $C$-strong for some positive number $C>0$. We compute 
\[
\tfrac{1}{f_\avg}\Psi^*\Omega_f-\Omega_\infty=\tfrac{1}{f_\avg}\Psi^*(\di\alpha_{\mathrm{can}})+\Psi^*(\mathfrak p_\infty^*(f_\norm \mu))-\mathfrak p_\infty^*\mu=\tfrac{1}{f_\avg}\Psi^*(\di\alpha_{\mathrm{can}}).
\]
Combining this identity with the bound for $\Vert\Psi^*(\di\alpha)\Vert_{C^2}$ found above, we arrive at
\begin{equation*}
\Big\Vert \tfrac{1}{f_\avg}\Psi^*\Omega_f-\Omega_\infty\Big\Vert_{C^2}=\tfrac{1}{f_\avg}\Vert\Psi^*(\di\alpha_{\mathrm{can}})\Vert_{C^2}\leq \frac{\big(\langle f\rangle_3^4+\langle f\rangle_2^6\big)e^{C_3\langle f\rangle^2_1}}{\big(\langle f\rangle_3^4+\langle f\rangle_2^6\big)e^{C\langle f\rangle^2_1}}= e^{(C_3-C)\langle f\rangle^2_1}\leq e^{C_3-C},
\end{equation*}
which can be made arbitrarily small, if $C$ is arbitrarily large. In particular, $\tfrac{1}{f_\avg}\Psi^*\Omega_f$ belongs to the given $C^2$-neighbourhood $\mathcal U$. 
\end{proof}

\subsection{A systolic-diastolic inequality for odd-symplectic forms}

The aim of this subsection is twofold. First, we give definitions and properties of the volume and the action of odd-symplectic forms. Then, we recall a local systolic-diastolic inequality for odd-symplectic forms on closed three-manifolds established in \cite{BK19}.

\subsubsection*{Weakly Zoll pairs}
Consider the space of all oriented $S^1$-bundles $\mathfrak p:\ta^1M\to M_{\mathfrak p}$ with total space $\ta^1 M$, where $M_{\mathfrak p}$ is some closed oriented surface (diffeomorphic to $M$). Let $\mathfrak P^0(\ta^1M)$ be the connected component of such a space containing $\mathfrak p_\infty:\ta^1M\to M$. A pair $(\mathfrak p,c)$, where $\mathfrak p\in\mathfrak P^0(\ta^1M)$ and $c\in H^2_\mathrm{dR}(M_{\mathfrak p})$ is called a weakly Zoll pair. A closed two-form $\Omega$ on $\ta^1M$ is said to be associated with $(\mathfrak p,c)$, if $\Omega=\mathfrak p^*\om$ for some closed two-form $\om$ on $M_{\mathfrak p}$ satisfying $[\om]=c$. As discussed above, every Zoll form $\Omega$ canonically defines a weakly Zoll pair $(\mathfrak p_\Omega,[\om])$. For example, the Zoll form $\Omega_\infty=\mathfrak p_\infty^*\mu$ is associated with the weakly Zoll pair  $(\mathfrak p_\infty,[\mu])$. 

Let $\mathfrak Z_{[\Omega_\infty]}^0(\ta^1M)$ be the set of all weakly Zoll pairs $(\mathfrak p,c)$ such that
\[
\mathfrak p\in\mathfrak P^0(\ta^1 M),\qquad \mathfrak p^*c=[\Omega_\infty]\in H^2_\dR(\ta^1M).
\] 
Below, we define and compute volume, action, and Zoll polynomial with respect to some fixed reference weakly Zoll pair 
\[
(\mathfrak p_\infty,c_0)\in\mathfrak Z_{[\Omega_\infty]}^0(\ta^1M).
\] 
As we specify in the next subsection, we take different reference pairs for $M\neq\T^2$ and for $M=\T^2$. This will enable us to simplify computations. However, as observed in \cite[Remark 1.17]{BK19}, a different choice results in different volume, action, and Zoll polynomial but in an equivalent systolic-diastolic inequality.
\subsubsection*{Volume}

We pick any closed form $\om_0$ on $M$ with $[\om_0]=c_0$ and set 
\[
\Omega_0=\mathfrak p_\infty^*\om_0.
\] 
Let $\Omega$ be a closed two-form on $\ta^1M$ with the same cohomology class as $\Omega_0$. We choose a one-form $\alpha$ on $\ta^1M$ such that $\Omega=\Omega_0+\di\alpha$. 
The volume of $\alpha$ is defined by
\[
\Vol(\alpha)=\frac{1}{2}\int_{\ta^1M}\alpha\wedge\di\alpha+\int_\Sigma\alpha\wedge\Omega_0.
\]
As seen in Section \ref{ss:unit}, $\mathfrak p_\infty^*:H^2_\mathrm{dR}(M)\to H^2_\mathrm{dR}(\ta^1M)$ vanishes when  $M\neq\T^2$, and thus $[\Omega_0]=0$. In this case $\Vol(\alpha')=\Vol(\alpha)$ for any $\alpha'$ satisfying $\di\alpha'=\di\alpha$. Therefore, we define the volume by
\[
\Fvol(\Omega)=\Vol(\alpha).
\]
By \cite[Proposition 2.8]{BK19}, if $\Psi$ is a diffeomorphism on $\ta^1M$ isotopic to the identity, then
\begin{equation}\label{eq:invariance_vol}
\Fvol(\Psi^*\Omega)=\Fvol(\Omega)
\end{equation}
If $M=\T^2$, then it can happen that $\di\alpha'=\di\alpha$ but $\Vol(\alpha')\neq\Vol(\alpha)$. In this case, we can choose $\alpha$ such that $\Vol(\alpha)=0$. Such a one-form is called {\bf normalised} and we declare
\[
\Fvol(\Omega)=0.
\]

\subsubsection*{Action}

We define the {action} on the space $\Lambda_{\mathfrak h_\infty}(\ta^1M)$ of one-periodic curves in the free homotopy class $\mathfrak h_\infty\in[S^1,\ta^1M]$ of $\mathfrak p_\infty$-fibres by 
\[
{\mathcal A}_\alpha:\Lambda_{\mathfrak h_\infty}(\ta^1M)\to\R,\qquad  \gamma\mapsto \int_{S^1}\gamma_0^*\alpha+\int_{[0,1]\x S^1}\Gamma^*\Omega.
\]
where $\Gamma:[0,1]\x S^1\to\ta^1M$ is any cylinder such that $\Gamma(1,\cdot)=\gamma$ and $\Gamma(0,\cdot)=\gamma_0$ is any oriented $\mathfrak p_\infty$-fibre. This action does not depend on the choice of $\om_0$ nor of $\Gamma$.
Moreover,  a critical point of $\mathcal A_\alpha$ is a closed characteristic of $\Omega$, i.e.~a closed curve tangent to the distribution $\ker\Omega$. We denote by $\mathcal X(\Omega)$ the set of embedded closed characteristics of $\Omega$.

In order to define the action with respect to $\Omega$, we observe that if $\alpha'$ is another one-form on $\ta^1M$ such that $\Omega=\Omega_0+\di\alpha'$, then
\[
\mathcal A_{\alpha'}=\mathcal A_\alpha+\int_{\mathfrak p_\infty^{-1}\rm(pt)}(\alpha'-\alpha)
\]
where $\mathfrak p_\infty^{-1}\rm(pt)$ is any fibre of $\mathfrak p_\infty$. 
 When $M\neq\T^2$, the homology class of $\mathfrak p_\infty^{-1}\rm(pt)$ is zero, and therefore we can simply set 
\[
\mathcal A_\Omega:=\mathcal A_\alpha.
\]
If $M=\T^2$, the actions $\mathcal A_\alpha$ and $\mathcal A_{\alpha'}$ might be different. Nevertheless it turns out that if $\alpha$ and $\alpha'$ have the same volume, they have the same action. In this case, we choose a normalised one-form $\alpha$, i.e.~$\Vol(\alpha)=0$ and set
\[
\mathcal A_\Omega:=\mathcal A_\alpha.
\]
In both cases, by \cite[Proposition 6.10]{BK19}, if $\Psi$ is a diffeomorphism on $\ta^1M$ isotopic to the identity, then
\begin{equation}\label{eq:invariance_action}
\mathcal A_{\Psi^*\Omega}(\gamma)=\mathcal A_{\Omega}(\Psi(\gamma)).
\end{equation}

%In order to see closed characteristics of $\Omega=\Omega_\alpha=\Omega_0+\di\alpha$, we consider the one-form $\mathfrak a=\mathfrak a(\Omega)$ on $\Lambda_{\mathfrak h}(\ta^1M)$ defined by
%\[
%\mathfrak a_\gamma(\xi)=\int_{S^1}\Omega(\xi(t),\dot\gamma(t))\di t,\qquad \forall \,\gamma\in\Lambda_{\mathfrak h}(\ta^1M),\;\forall \,\xi\in\ta_\gamma\Lambda_{\mathfrak h}(\ta^1M).
%\]

\subsubsection*{Zoll polynomial}

The {Zoll polynomial} $P:\R\to\R$  is defined by
\begin{equation}\label{eq:Zoll_polynomial}
P(A)=\langle e,[M]\rangle\frac{A^2}{2}+\langle c_0,[M]\rangle A.
\end{equation}

For $(\mathfrak p,c)\in\mathfrak Z^0_{[\Omega_\infty]}(\ta^1M)$, we choose any closed two-form $\om$ on $M_{\mathfrak p}$ with $[\om]=c$ and define the volume and the action of $(\mathfrak p,c)$ by
\begin{equation}\label{eq:vol_action_weak_Zoll}
\Fvol(\mathfrak p,c)=\Fvol(\mathfrak p^*\om),\qquad \mathcal A(\mathfrak p,c)=\mathcal A_{\mathfrak p^*\om}(\mathfrak p^{-1}({\rm pt})).
\end{equation}
Note that since $\mathcal A(\mathfrak p_\infty,c_0)=0$, there holds $\frac{\di P}{\di A}(\mathcal A(\mathfrak p_\infty,c_0))=\langle c_0,[M]\rangle$. More generally 
it is shown in \cite[Proposition 6.18]{BK19} that for any weakly Zoll pair $(\mathfrak p,c)\in \mathfrak Z^0_{[\Omega_\infty]}(\ta^1M)$
\begin{equation}\label{eq:derivative_Zoll_polynomial}
\frac{\di P}{\di A}(\mathcal A(\mathfrak p,c))=\langle c,[M_{\mathfrak p}]\rangle.
\end{equation}

The following result relates the action and the volume of a weakly Zoll pair through the Zoll polynomial. It can be thought as the equality case of the local systolic-diastolic inequality presented below. 
\begin{thm}{\cite[Theorem 1.14]{BK19}}\label{thm:Zoll_equality}
There holds
\[
P(\mathcal A(\mathfrak p,c))=\Fvol(\mathfrak p,c),\qquad \forall\,(\mathfrak p,c)\in\mathfrak Z^0_{[\Omega_\infty]}(\ta^1M).
\]
When $M=\T^2$, this is equivalent to $\mathcal A(\mathfrak p,c)=0$, $\forall\,(\mathfrak p,c)\in\mathfrak Z^0_{[\Omega_\infty]}(\ta^1M)$.
\end{thm}

\subsubsection*{The general inequality}
Let $\Omega_*$ be a Zoll form, which is associated with  a weakly Zoll pair $(\mathfrak p_1,c_1)\in \mathfrak Z^0_{[\Omega_\infty]}(\ta^1M)$. In our applications $\Omega_*$ will be either $\Omega_{f_*}$ for some Zoll function $f_*$ or $\Omega_\infty$. We fix a finite open covering $\{B_i\}$ of $M_1$ by balls so that all their pairwise intersections are contractible. Let $\Lambda(\mathfrak p_1)$ be the space of curves $\gamma\in\Lambda_{\mathfrak h_\infty}(\ta^1M)$ such that $\mathfrak p_1(\gamma)$ is contained in some $B_i$.
Abbreviating $\mathcal X(\Omega;\mathfrak p_1):=\mathcal X(\Omega)\cap\Lambda(\mathfrak p_1)$, we define
\[
\mathcal A_{\min}(\Omega):=\inf_{\gamma\in\mathcal X(\Omega;\mathfrak p_1)}\mathcal A_{\Omega}(\gamma),\qquad \mathcal A_{\max}(\Omega):=\sup_{\gamma\in\mathcal X(\Omega;\mathfrak p_1)}\mathcal A_{\Omega}(\gamma).
\]
By \cite[Section III]{Gin87}, if an odd-symplectic form $\Omega$ is such that $\Omega-\Omega_*$ is an exact $C^1$-close two-form, the set $\mathcal X(\Omega)\cap\Lambda(\mathfrak p_1)$ is compact and non-empty. Therefore, the numbers $\mathcal A_{\min}(\Omega)$ and $\mathcal A_{\max}(\Omega)$ are finite and they can be shown to vary $C^1$-continuously with $\Omega$. We finally state the local systolic-diastolic inequality for odd-symplectic forms.
\begin{thm}{\cite[Corollary 1.21]{BK19}}\label{thm:lsios}
There exists a $C^{2}$-neighbourhood $\mathcal U$ of $\Omega_*$ in the set of odd-symplectic forms on $\ta^1M$ with cohomology class $[\Omega_*]$ such that
\begin{equation*}
P(\mathcal A_{\min}(\Omega))\leq\Fvol(\Omega)\leq P(\mathcal A_{\max}(\Omega)),\qquad \forall\,\Omega\in \mathcal U.
\end{equation*}
Moreover the equality holds in any of the two inequalities exactly when $\Omega$ is Zoll. When $M=\T^2$, the inequality simplifies to 
\[
\mathcal A_{\min}(\Omega)\leq 0\leq \mathcal A_{\max}(\Omega).
\]
\end{thm}
\subsection{Volume and action of magnetic functions}\label{subsec:vol_action_magnetic}

\subsubsection*{Case $M\neq\T^2$}
As observed in Section \ref{ss:unit}, the two-forms $\Omega_\infty$ and $\Omega_f$, where $f:M\to\R$ is any function, are exact. Explicit primitives are given by 
\begin{equation*}\label{dfn:alphaf}
\alpha_\infty:= \frac{\area(M)}{\chi(M)}(\eta+\mathfrak p_\infty^*\zeta_\infty),\qquad \alpha_f:=\alpha_\can+\frac{\area(M)}{\chi(M)}(f_\avg\eta+\mathfrak p_\infty^*\zeta),
\end{equation*}
where $\zeta$ and $\zeta_\infty$ are one-forms on $M$ with differential 
\[
\di\zeta_\infty=\Big(\frac{\chi(M)}{\area(M)}-\frac{K}{2\pi}\Big)\mu,\qquad \di\zeta=\Big(\frac{\chi(M)}{\area(M)}f-\frac{f_\avg\cdot K}{2\pi}\Big)\mu.
\]
We choose as reference weakly Zoll pair 
\[
(\mathfrak p_\infty,c_0)=(\mathfrak p_\infty,0)\in \mathfrak Z_{[\Omega_\infty]}^0(\ta^1M).
\]
From formula \eqref{eq:Zoll_polynomial} and identity \eqref{e:state}, we have the Zoll polynomial
\begin{equation}\label{e:polneqt}
P(A)=\frac{\chi(M)}{2}A^2.
\end{equation}

Let $\Omega$ be an exact two-form on $\ta^1M$, and let $\alpha$ be an arbitrary primitive one-form of $\Omega$. In this case the volume of $\Omega$ is reduced to
\[
\Fvol(\Omega)=\frac{1}{2}\int_{\ta^1M}\alpha\wedge\Omega.
\]
and the action of $\Omega$ is given by
\[
 \mathcal A_{\Omega}(\gamma)=\int_{S^1}\gamma^*\alpha,\qquad \forall\,\gamma\in\Lambda_{\mathfrak h_\infty}(\ta^1M).
\]
We note that the volume is two-homogeneous while the action is one-homogeneous. Namely,
\begin{equation}\label{eq:homogeneity}
\Fvol(s\Omega)=s^2\Fvol(\Omega),\qquad \mathcal A_{s\Omega}=s\mathcal A_{\Omega},\qquad \forall\,s\in\R.
\end{equation}
\begin{lem}\label{l:actvolchineq0}
If $M\neq\T^2$ and $f:M\to \R$ is a function, we have
\begin{equation*}
\Fvol(\Omega_\infty)=\frac{\area(M)^2}{2\chi(M)},\qquad\Fvol(\Omega_f)=\frac{\area(M)^2}{2\chi(M)}K_f.
\end{equation*}
If $\gamma_0:S^1\to \ta^1 M$ is an oriented fibre of  $\mathfrak p_\infty$ and $c\in\Lambda(M; \mathfrak h_\infty)$, then
\begin{equation*}
\mathcal A_{\Omega_\infty}(\gamma_0)=\frac{\area(M)}{\chi(M)},\qquad \mathcal A_{\Omega_f}(c,\dot c)=\ell_f(c)+\frac{\area(M)\cdot f_\avg}{\chi(M)}.
\end{equation*}
\end{lem}
\begin{proof}
We compute the volume of $\Omega_\infty$ as
\[
\Fvol(\Omega_\infty)=\frac{1}{2}\int_{\ta^1M}\!\!\alpha_\infty\wedge \Omega_\infty=\frac{\area(M)}{2\chi(M)}\int_{\ta^1M}\!\!\eta\wedge\mathfrak p_\infty^*\mu= \frac{\area(M)}{2\chi(M)}\int_{M}((\mathfrak p_\infty)_*\eta)\mu= \frac{\area(M)^2}{2\chi(M)}.
\] 
To determine the volume of $\Omega_f$, we perform first the preliminary computation
\[
\alpha_f\wedge\Omega_f=\alpha_\can\wedge\di\alpha_\can+\frac{\area(M)}{\chi(M)}\big(f_\avg\eta\wedge\mathfrak p_\infty^*(f\mu)+\mathfrak p_\infty^*\zeta\wedge \di\alpha_\can\big),
\]
using the fact that $X$ annihilates $\eta\wedge\di\alpha$ and $V$ annihilates $\alpha\wedge\mathfrak p_\infty^*(f\mu)$. Then,
\[
\begin{split}
2\,\Fvol(\Omega_f)&=\int_{\ta^1 M}\alpha_f\wedge\Omega_f\\
&=\int_{\ta^1 M}\alpha_\can\wedge\di\alpha_\can+\frac{\area(M)}{\chi(M)}\left[\int_{\ta^1 M}f_\avg\eta\wedge\mathfrak p_\infty^*(f\mu)+\int_{\ta^1 M}\mathfrak p_\infty^*\zeta\wedge\di\alpha_\can\right]\\
&={2\pi}\int_{\ta^1 M}\eta\wedge\mathfrak p_\infty^*\mu+\frac{\area(M)}{\chi(M)}\left[f_\avg\int_{M}f\mu+\int_{\ta^1 M}\mathfrak p_\infty^*(\di\zeta)\wedge\alpha_\can\right]\\
&=2\pi\cdot\area(M)+\frac{\big(\area(M)\cdot f_\avg\big)^2}{\chi(M)}\\
&=\frac{\area(M)^2}{\chi(M)}K_f,
\end{split}
\]
where we used \eqref{eq:alpha_can} and the fact that $V$ annihilates $\mathfrak p_\infty^*(\di\zeta)\wedge\alpha_\can$.

Next we compute the actions. For the $\Omega_\infty$-action of $\gamma_0$ we find
\[
\mathcal A_{\Omega_\infty}(\gamma_0)=\int_{S^1}\gamma_0^*\alpha_\infty=\frac{\area(M)}{\chi(M)}\left(\int_{S^1}\gamma_0^*(\eta+\mathfrak p_\infty^*\zeta_\infty)\right)=\frac{\area(M)}{\chi(M)}.
\]
To compute the $\Omega_f$-action of $(c,\dot c)$, let $\Gamma:[0,1]\x S^1\to\ta^1M$ be a cylinder connecting an oriented $\mathfrak p_\infty$-fibre to $(c,\dot c)$ and recall the formula for the magnetic length \eqref{eq:length_Gamma}. Using Stokes' theorem we compute
\begin{align*}
\mathcal A_{\Omega_f}(c,\dot c)=\int_{\R/T\Z}(c,\dot c)^*\alpha_f=\int_{[0,1]\times S^1}\Gamma^*\Omega_f+\int_{S^1}\Gamma(0,\cdot)^*\alpha_f&=\ell_f(c)+\frac{\area(M)\cdot f_\avg}{\chi(M)},
\end{align*}
where in the last passage we used that $\int_{S^1}\Gamma(0,\cdot)^*\alpha_\can=0$.
\end{proof}
Let $f_*:M\to\R$ be a Zoll function, whose magnetic geodesics lie in $\Lambda(M;\mathfrak h_\infty)$, and let $(\mathfrak p_{f_*},[\om_{f_*}])$ be the weakly Zoll pair associated with the Zoll odd-symplectic form $\Omega_{f_*}$. Due to Proposition \ref{p:magzollintro}.(a), there holds 
\begin{equation}\label{e:zomegainfty}
(\mathfrak p_{f_*},[\om_{f_*}])\in\mathfrak Z_{[\Omega_\infty]}^0(\ta^1M).
\end{equation}
Therefore, from  \eqref{eq:derivative_Zoll_polynomial} and Theorem \ref{thm:Zoll_equality}, we have
\begin{equation}\label{e:omegafMf}
0<\langle [\omega_{f_*}],[M_{f_*}]\rangle=\frac{\di P}{\di A}(\mathcal A(\Omega_{f_*})), \qquad P(\mathcal A(\Omega_{f_*}))=\Fvol(\Omega_{f_*}),
\end{equation}
where $\mathcal A(\Omega_{f_*}):=\mathcal A(\mathfrak p_{f_*},[\om_{f_*}])$ and $\Fvol(\Omega_{f_*}):=\Fvol(\mathfrak p_{f_*},[\om_{f_*}])$ are the action and the volume defined in \eqref{eq:vol_action_weak_Zoll}. In our case, it reads
\[
\mathcal A(\Omega_{f_*})=\int_{S^1}(c_{f_*},\dot c_{f_*})^*\alpha_{f_*},
\]
where $c_{f_*}$ is a prime closed $f_*$-magnetic geodesic.
\begin{cor}\label{c:sign}
If $f_*:M\to\R$ is a Zoll function and $M\neq\T^2$, then
\begin{equation*}
\mathcal A(\Omega_{f_*})=\frac{\langle [\omega_{f_*}],[M_{f_*}]\rangle}{\chi(M)},\qquad K_{f_*}=\left(\frac{\chi(M)\mathcal A(\Omega_{f_*})}{\area(M)}\right)^2=\left(\frac{\langle [\omega_{f_*}],[M_{f_*}]\rangle}{\area(M)}\right)^2.
\end{equation*}
In particular, $\mathcal A(\Omega_{f_*})$ and $\Fvol(\Omega_{f_*})$ have the same sign as $\chi(M)$, and $K_{f_*}$ is positive.
\end{cor}
\begin{proof}
From \eqref{e:polneqt} we get $\tfrac{\di P}{\di A}(\mathcal A(\Omega_{f_*}))=\chi(M)\mathcal A(\Omega_{f_*})$, which together with the first relation in \eqref{e:omegafMf} yields the statement about $\mathcal A(\Omega_{f_*})$. Putting the second relation in \eqref{e:omegafMf}, equation \eqref{e:polneqt}, and Lemma \ref{l:actvolchineq0} together, we have
\[
\frac{\chi(M)}{2}\mathcal A(\Omega_{f_*})^2=P(\mathcal A(\Omega_{f_*}))=\Fvol(\Omega_{f_*})=\frac{\area(M)^2}{2\chi(M)}K_{f_*}.
\]
This proves the rest of the corollary.
\end{proof}
\subsubsection*{Case $M=\T^2$}
We work with the reference weakly Zoll pair
\[
(\mathfrak p_\infty,c_0)=(\mathfrak p_\infty,[\mu])\in\mathfrak Z_{[\Omega_\infty]}^0(\ta^1\T^2), 
\] 
so that $\Omega_0=\mathfrak p_\infty^*\mu=\Omega_\infty$. This form is not exact by the discussion in Section \ref{ss:unit}. Let $f:\T^2\to\R$ be an arbitrary function with $f_\avg>0$, so that
\[
\bar\ell(f)=\frac{\pi}{f_\avg}>0.
\]
We consider the normalised form
\begin{equation*}
\bar{\Omega}_f:=\frac{1}{f_\avg}\Omega_f
\end{equation*}
so that $\bar\Omega_f$ and $\Omega_\infty$ are cohomologous. More precisely,
\begin{equation*}\label{dfn:alphaftorus}
\bar\Omega_f=\Omega_\infty+\di\big(\tfrac{1}{f_\avg}\alpha_f\big),\qquad \alpha_f:=\alpha_\can+\mathfrak p_\infty^*\zeta-\bar\ell(f)\di\phi, 
\end{equation*}
where $\zeta$ is a one-form on $\T^2$ is such that $\di\zeta=(f-f_\avg)\mu$ and $\phi:\ta^1\T^2\to S^1$ is a global angular function for the bundle $\mathfrak p_\infty$, namely $\di\phi(V)\equiv 1$. As we see in the next lemma, the term $-\bar\ell(f)\di\phi$ is added in order to normalise $\frac{1}{f_\avg}\alpha_f$.

\begin{lem}\label{l:acvolchi0}
Let $f:\T^2\to\R$ be a function with $f_\avg>0$. Then, the one-form $\tfrac{1}{f_\avg}\alpha_f$ is normalised, i.e.~$\Vol(\tfrac{1}{f_\avg}\alpha_f)=0$.
\end{lem}
\begin{proof}
Using \eqref{eq:alpha_can}, we compute
\begin{align*}
(f_\avg)^2\Vol(\tfrac{1}{f_\avg}\alpha_f)&=f_\avg\int_{\ta^1\T^2}\alpha_f\wedge \Big(\Omega_\infty+\frac{1}{2}\di\big(\tfrac{1}{f_\avg}\alpha_f\big)\Big)\\
&=\int_{\ta^1\T^2}\big(\alpha_\can+\mathfrak p_\infty^*\zeta-\bar\ell(f)\di\phi\big)\wedge\Big(f_\avg\,\mathfrak p_\infty^*\mu+\frac{1}{2}\big(\di\alpha_\can+\mathfrak p_\infty^*(\di\zeta)\big)\Big)  \\
&=\frac{1}{2}\int_{\ta^1 \T^2}\alpha_\can\wedge\di\alpha_\can-\bar\ell(f)f_\avg\cdot\area(\T^2)\\
&=\pi\int_{\ta^1 \T^2}\eta\wedge\mathfrak p_\infty^*\mu-\pi\cdot\area(\T^2)\\
&=0.\qedhere
\end{align*}
\end{proof}
By Lemma \ref{l:acvolchi0}, we can use the one-form $\tfrac{1}{f_\avg}\alpha_f$ to compute the $\bar{\Omega}_f$-action of loops:
\begin{equation}\label{e:actionmagt2}
\mathcal A_{\bar\Omega_f}(\gamma)=\frac{1}{f_\avg}\int_{S^1}\Gamma(0,\cdot)^*\alpha_f+\frac{1}{f_\avg}\int_{[0,1]\x S^1}\Gamma^*\Omega_f,
\end{equation}
where $\Gamma:[0,1]\x S^1\to\ta^1\T^2$ is a homotopy between an oriented $\mathfrak p_\infty$-fibre and $\gamma\in\Lambda_{\mathfrak h_\infty}(\ta^1\T^2)$.

\begin{lem}\label{l:actorus}
Let $f:\T^2\to\R$ be a function with $f_\avg>0$. There holds 
\begin{equation*}
\mathcal A_{\bar\Omega_f}(c,\dot c)=\frac{1}{f_\avg}\big(\ell_f(c)-\bar\ell(f)\big),\qquad \forall\,c\in\Lambda(\T^2;\mathfrak h_\infty).
\end{equation*}
\end{lem}
\begin{proof}
The claim follows from substituting identity \eqref{eq:length_Gamma} in \eqref{e:actionmagt2}  and the computation
\[
\int_{S^1}\Gamma(0,\cdot)^*\alpha_f=\int_{S^1}\Gamma(0,\cdot)^*\big(\alpha+\mathfrak p_\infty^*\zeta-\bar\ell(f)\di\phi\big)=-\bar\ell(f).\qedhere
\]
\end{proof}
Finally, we observe that if $f_*:\T^2\to\R$ is a Zoll function whose magnetic geodesics lie in $\Lambda(\T^2;\mathfrak h_\infty)$, then, by Proposition \ref{p:magzollintro}.(b), we have
\begin{equation}\label{e:f*pos}
(f_*)_\avg>0
\end{equation}
and, setting $\bar{\om}_{f_*}:=\tfrac{1}{f_\avg}\om_{f_*}$, we see that $(\mathfrak p_{f_*},[\bar{\om}_{f_*}])$ is the weakly Zoll pair associated with the Zoll odd-symplectic form $\bar{\Omega}_{f_*}$. By Proposition \ref{p:magzollintro}.(a), there holds
\begin{equation}\label{e:zomegainftyt2}
(\mathfrak p_{f_*},[\bar{\om}_{f_*}])\in\mathfrak Z^0_{[\Omega_\infty]}(\ta^1\T^2).
\end{equation}

\section{The proof of the magnetic systolic-diastolic inequality}\label{sec:magsys}

\subsection{The inequality in a neighbourhood of a Zoll function}\label{ss:zollfun}
In this subsection we give a proof of Theorem \ref{t:mag2}, which states that the magnetic systolic-diastolic inequality holds in a $C^2$-neighbourhood $\mathcal F$ of a Zoll function $f_*:M\to \R$. As before we deal with the cases $M\neq \T^2$ and $M=\T^2$ separately.

\subsubsection*{Proof of Theorem \ref{t:mag2} for $M\neq\T^2$}
In view of \eqref{e:zomegainfty} and Theorem \ref{thm:lsios}, there exist a $C^1$-neighbourhood \label{dfn:mathcalw1}$\mathcal W$ of the set $\Lambda(f_*;\mathfrak h_\infty)$ in $\Lambda(M; \mathfrak h_\infty)$ and a $C^2$-neighbourhood $\mathcal F$ of the function $f_*$ in $C^\infty(M)$ such that 
\begin{equation*}\label{dfn:aminmaxf}
\mathcal A_{\min}(\Omega_f)=\min_{\substack{c\in\mathcal W\cap\Lambda(f;\mathfrak h_\infty)\\ c\,\text{prime}}}\mathcal A_{\Omega_f}(c,\dot c),\qquad \mathcal A_{\max}(\Omega_f)=\max_{\substack{c\in\mathcal W\cap\Lambda(f;\mathfrak h_\infty)\\ c\,\text{prime}}}\mathcal A_{\Omega_f}(c,\dot c).
\end{equation*}
and
\begin{equation}\label{e:sys2}
P(\mathcal A_{\min}(\Omega_{f}))\leq\Fvol(\Omega_{f})\leq P(\mathcal A_{\max}(\Omega_{f})),\qquad \forall\,f\in\mathcal F
\end{equation}
with equality signs if and only if $\Omega_{f}$ is Zoll. Since $\mathcal A_{\min}(\Omega_f)$ and $\mathcal A _{\max}(\Omega_f)$ vary continuously in $f\in\mathcal F$, shrinking $\mathcal F$ if necessary, we deduce from Corollary \ref{c:sign} that for all $f\in\mathcal F$:
\begin{equation}\label{e:prop123}
K_f> 0,\qquad \mathrm{sign}\big(\mathcal A_{\min}(\Omega_{f})\big)=\mathrm{sign}\big(\chi(M)\big)=\mathrm{sign}\big(\mathcal A_{\max}(\Omega_{f})\big).
\end{equation}

We show that the magnetic systolic-diastolic inequality holds on $\mathcal F$.
Let $f:M\to \R$ be a function in $\mathcal F$. According to Lemma \ref{l:actvolchineq0} and equation \eqref{e:polneqt}, formula \eqref{e:sys2} becomes
\[
\chi(M)\frac{\mathcal A_{\min}(\Omega_{f})^2}{2}\leq \frac{\area(M)^2}{2\chi(M)}K_{f}\leq \chi(M)\frac{\mathcal A_{\max}(\Omega_{f})^2}{2}.
\]
The identities in \eqref{e:prop123} simplify this inequality to 
\[
\mathcal A_{\min}(\Omega_{f})\leq \frac{\area(M)}{\chi(M)}\sqrt{K_f}\leq \mathcal A_{\max}(\Omega_{f}).
\]
The formula for the action in Lemma \ref{l:actvolchineq0} and the definition of $\ell_{\min}(f)$, $\ell_{\max}(f)$ yield
\[
\mathcal A_{\min}(\Omega_{f})\geq\ell_{\min}(f)+\frac{\area(M)\cdot f_\avg}{\chi(M)},\qquad  \mathcal A_{\max}(\Omega_{f})\leq \ell_{\max}(f)+\frac{\area(M)\cdot f_\avg}{\chi(M)},
\]
where the equalities hold when $f$ is Zoll. Combining the inequalities above, we get
\[
\ell_{\min}(f)\leq \frac{\area(M)}{\chi(M)}\big(\sqrt{K_f}-f_\avg\big)\leq \ell_{\max}(f),
\]
and using the definition of the average curvature, we rewrite the term in the middle as
\[
 \frac{\area(M)}{\chi(M)}\big(\sqrt{K_f}-f_\avg\big)=\frac{\area(M)}{\chi(M)}\frac{K_f-(f_\avg)^2}{\sqrt{K_f}+f_\avg}=\frac{2\pi}{\sqrt{K_f}+f_\avg}=\bar\ell(f).
\]
This shows the magnetic systolic-diastolic inequality $\ell_{\min}(f)\leq \bar\ell(f)\leq \ell_{\max}(f)$. Moreover, if $f$ is Zoll, we actually have equalities. Conversely, if one of the two inequalities is an equality, we also have an equality in \eqref{e:sys2}. This implies that $\Omega_{f}$, and thus $f$, is Zoll.
\qed

\subsubsection*{Proof of Theorem \ref{t:mag2} for $M=\T^2$}

Thanks to \eqref{e:f*pos}, \eqref{e:zomegainftyt2} and Theorem \ref{thm:lsios}, there exists a $C^1$-neighbourhood $\mathcal W$ of $\Lambda(f_*;\mathfrak h_\infty)$ inside $\Lambda(\T^2;\mathfrak h_\infty)$ and a $C^2$-neighbourhood $\mathcal F$ of $f_*$ in $C^\infty(\T^2)$ with the following properties. If $f\in \mathcal F$, then $f_\avg>0$ and 
\begin{equation}\label{e:sys3}
\mathcal A_{\min}(\bar\Omega_f)\leq 0\leq \mathcal A_{\max}(\bar\Omega_f),\qquad\forall\, f\in\mathcal F,
\end{equation}
where any of the two equalities holds if and only if $\bar\Omega_f$ is Zoll. Since $\Omega_f$ and $\bar{\Omega}_f$ have the same closed characteristics, we have
\[
\mathcal A_{\min}(\bar\Omega_f):=\min_{\substack{c\in\mathcal W\cap\Lambda(f;\mathfrak h_\infty)\\ c\,\text{prime}}}\mathcal A_{\bar\Omega_f}(c,\dot c),\qquad \mathcal A_{\max}(\bar\Omega_f):=\max_{\substack{c\in\mathcal W\cap\Lambda(f;\mathfrak h_\infty)\\ c\,\text{prime}}}\mathcal A_{\bar\Omega_f}(c,\dot c).
\]
From the definition of $\ell_{\min}(f)$ and $\ell_{\max}(f)$ and Lemma \ref{l:actorus}, we get
\[
\mathcal A_{\min}(\bar\Omega_{f})\geq \frac{1}{f_\avg}\big(\ell_{\min}(f)-\bar\ell(f)\big),\qquad 
\mathcal A_{\max}(\bar\Omega_{f})\leq \frac{1}{f_\avg}\big(\ell_{\max}(f)-\bar\ell(f)\big)
\]
where any of the two equalities holds, if $f$ is Zoll. Plugging these relations into \eqref{e:sys3}, and using that $f_\avg$ is positive, we derive the desired inequality:
\[
\ell_{\min}(f)\leq\bar\ell(f)\leq\ell_{\max}(f),
\]
If any of the equalities holds, then there is an equality also in \eqref{e:sys3} and $f$ is Zoll. The converse is also readily seen to be true.
\qed
\subsection{The inequality for strong magnetic functions}\label{ss:strong}
In this subsection we prove Theorem \ref{t:mag}, which states that the magnetic systolic-diastolic inequality holds for $C_g$-strong functions (see Definition \ref{d:strong}), where $C_g>0$ is a constant depending only on $g$ that we will determine.
	
\subsubsection*{Proof of Theorem \ref{t:mag} for $M\neq\T^2$}

By Theorem  \ref{thm:lsios}, there exists a $C^0$-neighbourhood $\Lambda(\mathfrak p_\infty)\subset\Lambda_{\mathfrak h_\infty}(\ta^1M)$ of the $\mathfrak p_\infty$-fibres and a $C^2$-neighbourhood $\mathcal U$ of $\Omega_\infty$ in the space of exact odd-symplectic forms on $\ta^1M$ such that 
\begin{equation}\label{e:msdistr}
P(\mathcal A_{\min}(\Omega))\leq\Fvol(\Omega)\leq P(\mathcal A_{\max}(\Omega)),\qquad\forall\, \Omega\in\mathcal U
\end{equation}
with equality signs if and only if $\Omega$ is Zoll. Here, $\mathcal A_{\min}(\Omega)$ and $\mathcal A_{\max}(\Omega)$ are the minimal and maximal action among the closed characteristics in the set $\mathcal X(\Omega;\mathfrak p_\infty)$. Since $\mathcal A_{\Omega_\infty}(\gamma_0)$ and $\chi(M)$ have the same sign by Lemma \ref{l:actvolchineq0}, and $\mathcal A_{\min}$, $\mathcal A_{\max}$ vary continuously in $\mathcal U$, we have, up to shrinking $\mathcal U$,
\[
\mathrm{sign}(\mathcal A_{\min}(\Omega))=\mathrm{sign}(\chi(M))=\mathrm{sign}(\mathcal A_{\max}(\Omega)),\qquad\forall\,\Omega\in\mathcal U.
\]
In particular, from \eqref{e:msdistr} and the formula for $P$, we also have
\begin{equation}\label{eq:sign_equality}
\mathrm{sign}(\Fvol(\Omega))=\mathrm{sign}(\chi(M)).
\end{equation}

We prove the theorem with $C_g:=C_\mathcal U$, the constant given by Lemma \ref{l:normal}. Let us consider a $C_g$-strong function $f:M\to(0,\infty)$, and let $\Psi:\ta^1 M\to\ta^1 M$ be a diffeomorphism isotopic to the identity such that $\tfrac{1}{f_\avg}\Psi^*\Omega_f\in\mathcal U$, whose existence is ensured by Lemma \ref{l:normal}. From the homogeneity \eqref{eq:homogeneity} of the volume and its invariance property \eqref{eq:invariance_vol}, we have
\[
\Fvol\big(\tfrac{1}{f_\avg}\Psi^*\Omega_f\big)=\big(\tfrac{1}{f_\avg}\big)^2\Fvol(\Omega_f).
\]
From the formula for the volume in Lemma \ref{l:actvolchineq0} and the relation \eqref{eq:sign_equality}, we see that $K_f>0$. Using the homogeneity of the action and formula \eqref{e:polneqt} for $P$, we can rewrite \eqref{e:msdistr} as
\[
\mathcal A_{\min}(\Psi^*\Omega_f)\leq \frac{\area(M)}{\chi(M)}\sqrt{K_f}\leq \mathcal A_{\max}(\Psi^*\Omega_{f}).
\]

Since $\Psi$ is isotopic to $\id_{\ta^1M}$, we also see that 
\[
c_\gamma:=\mathfrak p_\infty(\Psi(\gamma))\in\Lambda(f; \mathfrak h_\infty),\qquad \forall\, \gamma\in\mathcal X(\Psi^*\Omega_f;\mathfrak p_\infty).
\]
From Lemma \ref{l:actvolchineq0} and the invariance property \eqref{eq:invariance_action}, we conclude that
\[
\mathcal A_{\Psi^*\Omega_f}(\gamma)=\mathcal A_{\Omega_f}(c_\gamma,\dot c_\gamma)=\ell_f(c_\gamma)+\frac{\area(M)\cdot f_\avg}{\chi(M)},\qquad \forall\, \gamma\in\mathcal X(\Psi^*\Omega_f;\mathfrak p_\infty). 
\]
From the definition of $\ell_{\min}(f)$ and $\ell_{\max}(f)$, we have
\begin{align*}
\ell_{\min}(f)+\frac{\area(M)\cdot f_\avg}{\chi(M)}\leq \mathcal A_{\min}(\Psi^*\Omega_f),\qquad \mathcal A_{\max}(\Psi^*\Omega_f)\leq\ell_{\max}(f)+\frac{\area(M)\cdot f_\avg}{\chi(M)}
\end{align*}
and equalities hold if $f$ is Zoll. The rest of the proof goes along the same line as in the proof of Theorem \ref{t:mag2} for $M\neq \T^2$ in Section \ref{ss:zollfun} above.\qed

\subsubsection*{Proof of Theorem \ref{t:mag} for $M=\T^2$}

Theorem \ref{thm:lsios} yields a $C^0$-neighbourhood $\Lambda(\mathfrak p_\infty)$ of the $\mathfrak p_\infty$-fibres and a $C^2$-neighbourhood $\mathcal U$ of $\Omega_\infty$ in the space of odd-symplectic forms cohomologous to $\Omega_\infty$ such that
\begin{equation}\label{e:amin0amax}
\mathcal A_{\min}(\Omega)\leq 0\leq \mathcal A_{\max}(\Omega),\qquad\forall\,\Omega\in\mathcal U,
\end{equation}
and any of the equalities holds if and only if $\Omega$ is Zoll. We prove the theorem with $C_g:=C_\mathcal U$, the constant in Lemma \ref{l:normal}. Let $f:\T^2\to(0,\infty)$ be a $C_g$-strong function. In particular we have $f_\avg>0$. Let $\Psi$ be the diffeomorphism isotopic to the identity constructed in Lemma \ref{l:normal} with the property that $\Psi^*\bar\Omega_f\in\mathcal U$. Since $\Psi$ is isotopic to the identity, we see that $c_\gamma:=\mathfrak p_\infty(\Psi(\gamma))\in\Lambda(f; \mathfrak h_\infty)$ for all $\gamma\in\mathcal X(\Psi^*\bar\Omega_f;\mathfrak p_\infty)$. From \eqref{eq:invariance_action} and Lemma \ref{l:actorus}, we get 
\[
\mathcal A_{\Psi^*\bar\Omega_f}(\gamma)=\mathcal A_{\bar\Omega_f}(c_\gamma,\dot c_\gamma)=\frac{1}{f_\avg}(\ell_f(c_\gamma)-\bar{\ell}(f)),\qquad \forall\, \gamma\in\mathcal X(\Psi^*\bar\Omega_f;\mathfrak p_\infty).
\]
This relation together with \eqref{e:amin0amax} yields
\[
\ell_{\min}(f)-\bar\ell(f)\leq f_\avg\cdot\mathcal A_{\min}(\Psi^*\bar\Omega_f)\leq 0\leq f_\avg\cdot\mathcal A_{\max}(\Psi^*\bar\Omega_f)\leq \ell_{\max}(f)-\bar\ell(f)
\]
which in turn implies 
\[
\ell_{\min}(f)\leq \bar\ell(f)\leq \ell_{\max}(f).
\]
If $f$ is Zoll, the equalities hold. Conversely if $\ell_{\min}(f)$ or $\ell_{\max}(f)$ are equal to $\bar\ell(f)$, then there is an equality also in \eqref{e:amin0amax}, which yields that $\Psi^*\bar\Omega_f$, and hence $f$, is Zoll.\qed

\appendix
%\addcontentsline{toc}{part}{Appendix Some $C^k$-estimates on compact manifolds}
%\addtocontents{toc}{\protect\addvspace{2.25em plus 1pt}}

\section{$C^k$-estimate on the time-one map of a flow}\label{s:appendix}
For $h,k\in\N$, we define the polynomial
\[
B_{h,k}:\R^{k+1}\to \R,\qquad B_{h,k}(x)=\sum_{a\in I_{h,k}}x^a,
\]
where $x=(x_0,\,\cdots,x_k)$, $x^a:=x_0^{a_0}\cdot\ldots\cdot x_k^{a_k}$, and $I_{h,k}$ is the following set of multi-indices
\begin{equation*}
I_{h,k}:=\Big\{a=(a_0,\,\cdots,a_k)\in\N^{k+1}\ \Big|\ 0<\sum_{j=0}^k(j+1)a_j\leq h+k\Big\}.
\end{equation*}
If $\Psi:\mathcal M_1\to\mathcal M_2$ is a map between two Riemannian manifolds, we use the short-hand
\begin{equation*}\label{dfn:bhk}
B_{h,k}\big(\Vert\di\Psi\Vert\big):= B_{h,k}\big(\Vert\di\Psi\Vert_{C^0},\,\cdots,\Vert\di\Psi\Vert_{C^k}\big).
\end{equation*}
A straightforward computation shows that there is a constant $C_k>0$ (depending only on $\mathcal M_1$ and $\mathcal M_2$) such that for an $h$-form $\eta$ on $\mathcal M_2$,
\begin{equation}\label{eq:Ck_estimate}
\|\Psi^*\eta\|_{C^k}\leq C_kB_{h,k}(\|\di\Psi\|)\|\eta\|_{C^k}.
\end{equation}

Let $\Psi=\Phi_1$ be the time-one map of the flow of a (time-dependent) vector field. Then, using Gronwall's Lemma inductively, one can estimate $B_{h,k}(\Vert\di\Phi_1\Vert)$ in terms of the vector field. Here, we give only the bound for $(h,k)=(2,2)$ which is what is needed in Lemma \ref{l:normal}. 
\begin{lem}\label{l:diff2}
For every compact manifold $\mathcal M$, there exists a constant $C_{k}>0$ with the following property. For every time-dependent vector field $X=\{X_s\}_{s\in[0,1]}$ on $\mathcal M$ such that the corresponding flow $\{\Phi_s\}$ is defined up to time $1$, there holds
\begin{equation*}
B_{2,2}\big(\Vert \di\Phi_1\Vert\big)\leq \Big(\langle \nabla X\rangle_{C^2}+\langle \nabla X\rangle_{C^1}^2\Big)e^{C\langle \nabla X\rangle_{C^0}},
\end{equation*}
where we have set $\displaystyle\langle \nabla X\rangle_{C^k}:=1+\max_{s\in[0,1]}\Vert \nabla X_s\Vert_{C^k},\ \forall\, k\in\N$.
\end{lem}
\begin{proof}
We preliminarily observe that if $V$ is a finite-dimensional vector space endowed with a norm coming from a scalar product, then for every $s\mapsto v(s)\in V$, there holds
\[
\frac{\di |v|}{\di s}\leq \Big|\frac{\di v}{\di s}\Big|.
\]
By the compactness of $\mathcal M$, we just need to prove the lemma in local coordinates. Recalling that $\p_s\Phi_s=X_s\circ \Phi_s$ by definition, we compute
\begin{align*}
\p_s\big|\di\Phi_s\big|&\leq\big|\di(X_s\circ\Phi_s)\big|\leq\big|\nabla X_s\big|\cdot\big|\di\Phi_s\big|,\\
\p_s\big|\nabla\di\Phi_s\big|&\leq\Big|\nabla\Big(\big(\nabla_{\di\Phi_s} X_s\big)_{\Phi_s}\Big)\Big|\leq \big|\nabla^2 X_s\big|\cdot\big|\di\Phi_s\big|^2+\big|\nabla X_s\big|\cdot\big|\nabla\di\Phi_s\big|,\\
\p_s\big|\nabla^2\di\Phi_s\big|&\leq\Big|\nabla\Big(\big(\nabla_{\di \Phi_s} \nabla_{\di\Phi_s}X_s\big)_{\Phi_s}+\big(\nabla_{\nabla\di\Phi_s} X_s\big)_{\Phi_s}\Big)\Big|\\
&\leq \big|\nabla^3 X_s\big|\cdot\big|\di\Phi_s\big|^3+3\big|\nabla^2 X_s\big|\cdot\big|\nabla\di\Phi_s\big|\cdot\big|\di\Phi_s\big|+\big|\nabla X_s\big|\cdot\big|\nabla^2\di\Phi_s\big|.
\end{align*}
We now apply Gronwall's Lemma \cite{Gro19} and indicate with $C>0$ a constant depending on $\mathcal M$ but not on $X$. Below, we can always bring the constant to the exponent because, by definition, $\langle \nabla X\rangle_{C^0}\geq 1$. Thus, we find that
\begin{align*}
\Vert \max_s \di\Phi_s\Vert&\leq e^{C\langle \nabla X\rangle_{C^0}},\\
\Vert \max_s \nabla\di\Phi_s\Vert&\leq \langle \nabla X\rangle_{C^1}\Vert \max_s\di\Phi_s\Vert^2_{C^0}e^{C\langle \nabla X\rangle_{C^0}}\leq \langle \nabla X\rangle_{C^1}e^{C\langle \nabla X\rangle_{C^0}},\\
\Vert \max_s \nabla^2\di\Phi_s\Vert&\leq \Big(\langle\nabla X\rangle_{C^2}\Vert \max_s\di\Phi_s\Vert^3_{C^0}+\langle\nabla X\rangle_{C^1}\Vert \max_s\nabla\di\Phi_s\Vert_{C^0}\Vert \max_s\di\Phi_s\Vert_{C^0}\Big)e^{C\langle \nabla X\rangle_{C^0}}\\
&\leq \big(\langle \nabla X\rangle_{C^2}+\langle \nabla X\rangle_{C^1}^2\big)e^{C\langle \nabla X\rangle_{C^0}}.
\end{align*}
Finally, from the definition of $B_{2,2}$, we get
\begin{align*}
B_{2,2}\big(\Vert\di\Phi_1\Vert\big)&\leq \Big[\sum_{a_1+2a_2+3a_3\leq 4}\langle \nabla X\rangle_{C^1}^{a_2}\big(\langle \nabla X\rangle_{C^2}+\langle \nabla X\rangle_{C^1}^2\big)^{a_3}\Big]e^{C\langle \nabla X\rangle_{C^0}}\\
&\leq \big(\langle \nabla X\rangle_{C^2}+\langle \nabla X\rangle_{C^1}^2\big)e^{C\langle \nabla X\rangle_{C^0}}.\qedhere
\end{align*}
\end{proof}

\newpage

\bibliographystyle{amsalpha}
\bibliography{systolic_bib_2}
\end{document}